\documentclass[letter,12pt]{amsart}
\usepackage{extsizes}
\usepackage{blindtext}
\usepackage[utf8]{inputenc}
\usepackage{amsmath}
\usepackage{amssymb}
\usepackage{amsthm}
\usepackage{pst-node}
\usepackage[margin=1in]{geometry}
 \usepackage{float}
 \usepackage{amssymb}
\usepackage{tikz-cd}
\usepackage{amscd}
\usepackage[all,color]{xy}
\usepackage{sseq}
\usepackage{amsfonts}
\usepackage{enumerate}
\usepackage{graphicx}
\usepackage{fancyhdr}
\usepackage{tikz}
\usepackage{graphics}
\usepackage{hyperref}
\usepackage[title]{appendix}
\usepackage[numbers]{natbib}
\usepackage{mathrsfs}

\theoremstyle{plain} \numberwithin{equation}{section}

\newtheorem{theorem}{Theorem}[section]

\newtheorem{corollary}[theorem]{Corollary}

\newtheorem{lemma}[theorem]{Lemma}

\newtheorem{proposition}[theorem]{Proposition}
\theoremstyle{definition}
\newtheorem{definition}[theorem]{Definition}

\newtheorem{remark}[theorem]{Remark}

\def \f-{f^{-1}}
\def \fp-{f^{-1}_\partial}

\usepackage{hyperref}
\usepackage{caption}
\usepackage{subcaption}

\title{On unifying Jones representations of braid groups and Thompson's group}
\author{Yangxiao Luo, Shunyu Wan}
\date{}
\usepackage{natbib}
\usepackage{graphicx}

\begin{document}

\begin{abstract}

We unify Jones’ representations of braid groups and the oriented Thompson's group by introducing the \emph{oriented dyadic braid group}. This finitely generated group contains every braid group and the oriented Thompson's group. We construct an embedding of the oriented Thompson's group into this group such that taking braid closures recovers Jones’ construction of oriented links. We also introduce the dyadic Temperley–Lieb algebra and construct a unitary representation of the oriented dyadic braid group extending the classical Jones representations of braid groups. Along the Thompson group's embedding, its distinguished matrix coefficient agrees with the corresponding coefficient of Jones’ representation of the oriented Thompson's group.
\end{abstract}

\maketitle

\section{Introduction}


The Jones polynomial was introduced by Vaughan Jones as a polynomial invariant of links in his seminal work \cite{MR766964}. Its original construction uses representations of braid groups in Temperley–Lieb algebras. Let $Br_n$ be the $n^{th}$ braid group, and $TL_n$ be the $n^{th}$ Temperley-Lieb algebra. Jones representations of braid groups take the form
$$
J_n: Br_n \to TL_n.
$$

Applying a trace and an appropriate normalization yields the Jones polynomial of the braid closure.

In more recent work \cite{MR3589908}, Jones introduced unitary representations of Thompson's group $F$ and its oriented subgroup $\vec{F}$. His construction associates links to elements of Thompson's group, with the Jones polynomial appearing as a matrix coefficient \cite{MR3924840} \cite{MR4020028} \cite{MR4294116}. Moreover, every link can be obtained from a Thompson's group element \cite{MR3589908} \cite{MR4071378}. These results suggest that Thompson’s groups can play a role analogous to that of braid groups in knot theory.

Connections between braid groups and Thompson groups have been studied through parenthesized braid groups \cite{MR2258261} \cite{MR2176650} \cite{MR1470726} \cite{MR1669949}, braided Thompson's groups \cite{MR2364825} \cite{MR2217649} \cite{MR4441306} \cite{MR4817766}, and larger groups extending both braid groups and
Thompson’s groups
\cite{MR2105950} \cite{MR1090167} \cite{MR2181963}. Here we seek a connection that also respects Jones' representations and the associated construction of links.

More precisely, we ask whether Jones’ two representations reveal a relationship between Thompson’s group and braid groups, and whether the two representations can be unified. To address these questions, we introduce the oriented dyadic braid group $B_{\vec{F}}$.

\begin{definition} \label{def:oriented-dyadic-braid-group}
    The \emph{oriented dyadic braid group} $B_{\vec{F}}$ has braid generators $\sigma_{ab}$ for any two dyadic rational numbers $a<b$ in $(0,1)$, and $\vec{F}$ generators $g_i$ for $i \in \mathbb{Z}_{\geq 0}$, with the following relations:
    \begin{align*}
        \text{braid relations: }&\sigma_{ab}\sigma_{cd}=\sigma_{cd}\sigma_{ab} \ \text{ if $(a-c)(a-d)(b-c)(b-d)>0$,} \\
        &\sigma_{ab}\sigma_{bc}=\sigma_{ac}\sigma_{ab}=\sigma_{bc}\sigma_{ac} \ \text{  for all $a,b,c$ with $a<b<c$}. \\ \text{$\vec{F}$ relations: } &g_n g_l = g_l g_{n+2} \ \text{ for all $l<n$} \\ \text{braid $\vec{F}$ relations: } &g_i\sigma_{ab} = \sigma_{g_i(a)g_i(b)}g_i \ \text{ for all $i,a \text{ and } b$} 
    \end{align*}
\end{definition}

The braid relations used here come from the Birman–Ko–Lee presentation \cite[Proposition 2.1]{MR1654165}, rather than the standard Artin presentation. In Section \ref{sec:BF-as-an-F-equivariant-direct-limit}, we will give a pictorial description of $B_{\vec{F}}$, and show that each $Br_n$ can be naturally embedded in $B_{\vec{F}}.$ As classical braids, we can naturally take closure of oriented dyadic braids. More specifically, there is a well-defined surjective map $\mathfrak{cl}: B_{\vec{F}}\to C_0$ (Definition \ref{def: closed up}), where $C_0$ is the set of links up to adding distant unlinked unknots.


Jones constructed a unitary representation $\pi_L$ of $\vec{F}$ on a certain vector space $\mathcal{V}$, where $\mathcal{V}$ has an inner product $\langle \cdot, \cdot \rangle$ valued in $C_0$, and a canonical element $\xi \in \mathcal{V}$. With this representation, a link class $L(g) \in C_0$ arises as the $\xi$-coefficient of $\pi_L(g)$:
$$
\langle \pi_L(g)(\xi), \xi \rangle = L(g)
$$

We have the following theorem.

\begin{theorem} \label{thm: embedding-preserves-link}
    There is a group embedding $\iota: \vec{F} \to B_{\vec{F}}$ such that for any $g \in \vec{F}$ we have $L(g)=\mathfrak{cl}(\iota(g))$.
\end{theorem}

By modifying the representation $\pi_L$ to another representation $\pi_J$, the $\xi$-coefficient of $\pi_J(g)$ can give rise to the Jones polynomial of $L(g)$. More precisely, there is a function $\langle \pi_J(\cdot)(\xi), \xi \rangle: \vec{F} \to \mathbb{Q}(A)$, the field of rational functions in $A$, such that for any $g \in \vec{F}$,
\begin{equation} \label{eq:thompson-rep-to-jones-polynomial}
    \langle \pi_J(g)(\xi), \xi \rangle = \textup{normalized Jones polynomial of } L(g)
\end{equation}

In Section \ref{sec:extending-braid-representations-to-dyadic-braids}, we define the oriented dyadic Temperley-Lieb algebra $TL_{\vec{F}}$, and an inner product on $TL_{\vec{F}}$ valued in $\mathbb{Q}(A)$. We construct a unitary representation $J_{\vec{F}}: B_{\vec{F}} \to TL_{\vec{F}}$ (Definition \ref{def:jones-dyadic-braid-representation} and Theorem \ref{thm:jones-dyadic-braid-representation-is-unitary}), extending the classical Jones representation of $Br_n$ such that $J_{\vec{F}}|_{Br_n} = J_n$. Specifically, let $\epsilon$ be the unit in $TL_{\vec{F}}$, then we have that for any $b \in B_{\vec{F}}$,
\begin{equation} \label{eq:dyadic-braid-rep-to-kauffman-bracket}
     \langle J_{\vec{F}}(b)(\epsilon), \epsilon \rangle = \textup{normalized Kauffman bracket of } \mathfrak{cl}(b)
\end{equation}

\begin{remark}
    In Eq. \ref{eq:thompson-rep-to-jones-polynomial}, the normalization depends on the element $g\in \vec{F}$, see Eq. \ref{equ: normalized Jones' polynomial for Thompson} for the precise normalization. In particular, this normalized Jones polynomial is a link invariant up to normalization. Similarly, for Eq. \ref{eq:dyadic-braid-rep-to-kauffman-bracket} the normalization depends on the element $b \in B_{\vec{F}}$, see Eq. \ref{eq: normalized Kauffman for B_F} for the precise normalization.
\end{remark}

By combining Theorem \ref{thm: embedding-preserves-link}, Eq. \ref{eq:thompson-rep-to-jones-polynomial} and Eq. \ref{eq:dyadic-braid-rep-to-kauffman-bracket}, we have the following theorem unifying Jones representations of braid groups and Jones representation of Thompson's group.

\begin{theorem} \label{thm:braid-rep-thompson-rep}
    Let $\pi_J$ be Jones representation of the oriented Thompson's group $\vec{F}$, and let $J_{\vec{F}}$ be the representation of $B_{\vec{F}}$ extending Jones representations of braid groups as above. Then for any element $g \in \vec{F}$, we have
    $$
    \langle \pi_J(g)(\xi), \xi \rangle = \langle J_{\vec{F}}(\iota(g))(\epsilon), \epsilon \rangle
    $$
\end{theorem}

 Recall that the infinite braid group $B_\infty$  also contains each $Br_n$ as a direct limit of braid groups \cite{MR19087} \cite{MR19088}, and $B_\infty$ cannot be finitely generated. In contrast, for $B_{\vec{F}}$ we have the following property.

\begin{theorem} \label{thm: B_F is finitely generated}
    $B_{\vec{F}}$ is finitely generated.
\end{theorem}

The study of the dyadic braid group and the dyadic Temperley-Lieb algebra relies heavily on the theory of concentration structures and $G$-equivariant direct limits, which are briefly introduced in Section \ref{sec:concentration-structures-on-categories} and Section \ref{sec:g-equivariant-direct-limits}. We refer the readers to \cite{MR5118516} for a more thorough treatment. Section \ref{sec:concentration-structures-on-categories} and Section \ref{sec:g-equivariant-direct-limits} are largely self-contained and may be of interest to readers interested in category theory.

\subsection*{Acknowledgement}
YL was supported in part by NSF grant DMS-2405044 to Slava Krushkal, and by The Ohio State University postdoc funding. SW was partially supported by Georgia Tech postdoc funding.

\subsection*{AI Statement} GPT-6 Astra was used to enhance readability and for proofreading. The agent also pointed out an error of normalization factor in an earlier version of the manuscript. The authors take full responsibility for the mathematical arguments and results.

\section{Background: Jones representations of braid groups and Thompson's group}

\subsection{Jones' representations of braid groups}

The $n^{th}$ braid group $Br_n$ has $n-1$ generators $\sigma_1, \sigma_2, \cdots, \sigma_{n-1}$, with the following Artin braid relations:

\begin{enumerate}
    \item $\sigma_i\sigma_j = \sigma_j \sigma_i$ if $|i-j|>1$,
    \item $\sigma_i\sigma_{i+1}\sigma_i = \sigma_{i+1}\sigma_i\sigma_{i+1}$ for $1 \leq i \leq n-2$.
\end{enumerate}

In the above Artin presentation of $Br_n$, the generator $\sigma_i$ is the braid with a single crossing
between its $i^{th}$ endpoint and its $(i+1)^{th}$ endpoint, and with trivial strands elsewhere.

There is another presentation of $Br_n$ with $\frac{n(n-1)}{2}$ generators $\sigma_{ab}$ for $1 \leq a < b \leq n$, with the following Birman-Ko-Lee braid relations \cite{MR1654165}:

\begin{enumerate}
    \item $\sigma_{ab}\sigma_{cd}=\sigma_{cd}\sigma_{ab}$ \ if $(a-c)(a-d)(b-c)(b-d)>0$,

    \item $\sigma_{ab}\sigma_{bc}=\sigma_{ac}\sigma_{ab}=\sigma_{bc}\sigma_{ac}$ \  for all $a,b,c$ with $a<b<c$.
\end{enumerate}

In the Birman-Ko-Lee presentation of $Br_n$, the generator $\sigma_{ab}$ is the braid obtained by first drawing a crossing between its $a^{th}$ endpoint and its $b^{th}$ endpoint, then adding trivial strands elsewhere. We require that the trivial strands between the $a^{th}$ endpoint and the $b^{th}$ endpoint are below the two strands we first drew. Note that $\sigma_{ab}$ can contain more than one crossing, unlike the Artin generators $\sigma_i$.

\vspace{5mm}

Let $A$ be a formal variable in the following sections.

\begin{definition}
    The $n^{th}$ Temperley-Lieb algebra $TL_n$ is the free $\mathbb{Z}[A, A^{-1}]$-module generated by all $(n,n)$-planar tangle diagrams, with the relation $\bigcirc = -A^2 - A^{-2}$. The multiplication in $TL_n$ is induced by concatenating $(n,n)$-planar tangles.
\end{definition}

Note that $TL_n$ is a unital associative algebra. The unit is the $(n,n)$-planar tangle diagram $id_n$ with $n$ trivial strands.

As a $\mathbb{Z}[A, A^{-1}]$-module, $TL_n$ is freely generated by all $(n,n)$-planar tangle diagrams with no circles. The dimension of $TL_n$ is the Catalan number $\frac{(2n)!}{n!(n+1)!}$. In particular, $TL_0$ has dimension $1$. If we identify the empty tangle diagram with $1$, then we have $TL_0 = \mathbb{Z}[A, A^{-1}]$.

Similar to braid groups, Temperley-Lieb algebras have a purely algebraic description. It turns out that the $n^{th}$ Temperley-Lieb algebra $TL_n$ is the unital associative $\mathbb{Z}[A, A^{-1}]$-algebra generated by $n-1$ generators $e_1, e_2,..., e_{n-1}$, with the following relations:
\begin{enumerate}
    \item $e_i^2 = (-A^2 - A^{-2}) e_i$ for $1 \leq i \leq n-1$.
    \item $e_i e_{i+1} e_i = e_i$ for $1 \leq i \leq n-2$.
    \item $e_i e_{i-1} e_i = e_i$ for $2 \leq i \leq n-1$.
    \item $e_i e_j = e_j e_i$ for $1 \leq i, j \leq n-1$ and $|i-j| > 1$.
\end{enumerate}

The generator $e_i$ is the planar tangle with an upper turnback and a lower turnback between its $i^{th}$ endpoint and its $(i+1)^{th}$ endpoint, and with trivial strands elsewhere (see Figure \ref{fig:tl-generator}).

\begin{figure}
    \centering
    \includegraphics[width=0.4\linewidth]{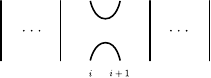}
    \caption{The generator $e_i$ of the Temperley-Lieb algebra.}
    \label{fig:tl-generator}
\end{figure}

\begin{definition}
    The trace map on $TL_n$ is the $\mathbb{Z}[A, A^{-1}]$-module homomorphism $tr: TL_n \to TL_0$, induced by taking closures of $(n,n)$-planar tangles.
\end{definition}

\bigskip

In the following definition, we regard the Temperley-Lieb algebra $TL_n$ as a monoid, by considering its multiplication structure only.

\begin{definition}
    The Jones representation of the braid group $Br_n$ is the monoid homomorphism $J_n: Br_n \to TL_n$ sending $\sigma_i$ to $A (id_n) + A^{-1} e_i$, sending $\sigma_i^{-1}$ to $A^{-1}(id_n) + A e_i$.
\end{definition}

$J_n(b)$ is also called the Kauffman bracket of $b$, denoted as $\langle b \rangle$. The Kauffman bracket
of a single crossing $\sigma_1 \in Br_2$ can be visualized in Figure \ref{fig:kauffman-relation}, which is sometimes
called the Kauffman relation.

\begin{figure}
    \centering
    \includegraphics[width=0.5\linewidth]{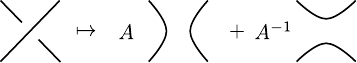}
    \caption{The Kauffman relation.}
    \label{fig:kauffman-relation}
\end{figure}

Let $b$ be a braid in $Br_n$. Define the bracket polynomial of $b$ to be the Laurent polynomial
$$
\langle b \rangle(A) = tr(J_n(b)) \in TL_0 = \mathbb{Z}[A, A^{-1}]
$$

Since $\langle \bigcirc \rangle(A) = -A^2 - A^{-2}$ and the closure of an $(n,n)$-planar tangle ($n \geq 1$) contains at least one circle, $\langle b \rangle(A)$ must be divisible by $A^2 + A^{-2}$.

Let $q = A^{-2}$, then $\langle b \rangle(q^{-\frac{1}{2}})$ is a Laurent polynomial in $\mathbb{Z}[q^{\frac{1}{2}}, q^{-\frac{1}{2}}]$ divisible by $q+q^{-1}$.

\begin{definition} \label{def:jones-polynomial}
    Let $b$ be a braid in $Br_n$. Define the Jones polynomial of $b$ to be the Laurent polynomial
    $$
    V_b(q) = (-1)^{|b|+1} q^{\frac{3w(b)}{2}}\frac{\langle b \rangle(q^{-\frac{1}{2}})}{q+q^{-1}} \in \mathbb{Z}[q, q^{-1}]
    $$
    where $|b|$ is the length of any braid word of $b$, and $w(b)$ is the writhe of the corresponding oriented tangle diagram. By the braid relations, $(-1)^{|b|+1}$ and $w(b)$ are both well-defined.
\end{definition}

Note that each crossing gives a coefficient $A = q^{-\frac{1}{2}}$ or $A^{-1} = q^{\frac{1}{2}}$, and the writhe $w(b)$ has the same parity as $|b|$. Hence the overall coefficient $q^{\frac{3w(b)}{2}}$ in $V_b(q)$ guarantees that $V_b(q)$ is in $\mathbb{Z}[q, q^{-1}]$, not just in $\mathbb{Z}[q^{\frac{1}{2}}, q^{-\frac{1}{2}}]$.

The Kauffman bracket can be generalized to a link diagram as follows. Given a link diagram $L$ with $|L|$ crossings, we resolve all crossings using the Kauffman relation (Figure \ref{fig:kauffman-relation}). There are totally $2^{|L|}$ ways to resolve $L$, and each resolution gives a set of circles with a coefficient $A^k$ for some $k \in \mathbb{Z}$. By assigning a factor $-A^2-A^{-2}$ to each circle and taking the sum the $2^{|L|}$ terms, we obtain the Kauffman bracket of $L$, denoted as $\langle L \rangle (A)$.

Then define the Jones polynomial of $L$ to be
$$
V_L(q) = (-1)^{|L|+1} q^{\frac{3w(L)}{2}}\frac{\langle L \rangle(q^{-\frac{1}{2}})}{q+q^{-1}} \in \mathbb{Z}[q, q^{-1}]
$$

If we take a braid $b$ with closure $L$, then this definition coincides with Definition \ref{def:jones-polynomial}.

We sometimes consider the Jones polynomial as a polynomial of $A$ instead, in which case the Jones polynomial can also be defined as follows:
    $$
    V_L(A) = (-A)^{-3w(L)} \frac{\langle L \rangle(A)}{-A^2-A^{-2}} \in \mathbb{Z}[A^2, A^{-2}].
    $$






\subsection{Jones' representation of Thompson's group}\label{sec: Jones' Thompson rep}

We first give a brief introduction to Thompson's group $F$ and related concepts.

\begin{definition} \label{def: Thompson group}
	Thompson's group $F$ is the group (under composition) of homeomorphisms $g$ from $[0,1]$ to itself satisfying the following conditions:	
	\begin{enumerate}
		\item $g$ is piecewise linear and orientation-preserving.
		\item In the pieces where $g$ is linear, the slope is always a power of $2$.
		\item The breakpoints are dyadic rational numbers in $[0,1]$, i.e. $\frac{k}{2^m}$ for some $m \in \{0,1,2,...\}$ and $k \in \{0,1,2,...,2^m\}$.
	\end{enumerate}
\end{definition}

\begin{definition} \label{s.d. interval and partition}
	A standard dyadic interval (s.d. interval) is an interval of the form $[\frac{k}{2^m}, \frac{k+1}{2^m}]$ for some $m \in \{0,1,2,...\}$ and $k \in \{0,1,2,...,2^m-1\}$. A standard dyadic partition (s.d. partition) on $[0,1]$ is a finite partition $\{0 = a_0 < a_1 < ... < a_n = 1\}$ such that $[a_i, a_{i+1}]$ is an s.d. interval for all $i$.

	We call $a_i$ a breakpoint of the s.d. partition, and call $[a_i, a_{i+1}]$ a subinterval of the s.d. partition. A standard dyadic partition with $n$ subintervals is called a standard dyadic $n$-partition (s.d. $n$-partition). Suppose $\mathcal{K}$ and $\mathcal{I}$ are s.d. partitions, we denote $\mathcal{K} \leq \mathcal{I}$ if $\mathcal{I}$ is a refinement of $\mathcal{K}$ as partitions. In other words, the set of breakpoints of $\mathcal{K}$ is a subset of the set of breakpoints of $\mathcal{I}$.

    Let $\mathfrak{T}$ denote the set of all standard dyadic partitions, then $(\mathfrak{T}, \leq)$ is an ordered set.
\end{definition}

\begin{definition} \label{def: s.d. partition pair representation}
	Suppose that $\mathcal{I}$ is a s.d. $n$-partition with breakpoints $0 = a_0 < a_1 < ... < a_n = 1$, and $\mathcal{J}$ is a s.d. $n$-partition with breakpoints $0 = b_0 < b_1 < ... < b_n = 1$. We define $\gamma^{\mathcal{I}}_{\mathcal{J}}: [0,1] \to [0,1]$ to be the piecewise linear function determined by letting $\gamma^{\mathcal{I}}_{\mathcal{J}}(a_i)=b_i$ for all $i \in \{0,1,...,n\}$ and letting $\gamma^{\mathcal{I}}_{\mathcal{J}}$ be linear on $[a_i,a_{i+1}]$ for all $i$.
	
	Note that $\gamma^\mathcal{I}_\mathcal{J}$ is an element in $F$. We say $(\mathcal{I}, \mathcal{J})$ is a pair of s.d. partitions representing $\gamma^\mathcal{I}_\mathcal{J}$. It is well known that every element in $F$ can be represented by some pair of s.d. partitions.
\end{definition}

\begin{definition}
    Given a s.d. partition $\mathcal{I} = \{0=a_0 < a_1 < \cdots < a_n = 1\}$, denote $\mathcal{E}(\mathcal{I}) = \{0=a_0<\frac{a_0+a_1}{2}<a_1 < \cdots < a_{n-1} < \frac{a_{n-1}+a_n}{2}\}$, that is, the union of endpoints and midpoints of $\mathcal{I}$ except $1$.
\end{definition}

Next we introduce Jones' representation of Thompson's group very briefly. For a detailed discussion, see \cite{MR3589908} \cite{MR3986040}. Roughly speaking, Jones' theory takes a planar algebra \cite{MR4374438} $C = (C_n)$ and a substance $R \in C_2$ as inputs. Given a s.d. $n$-partition $\mathcal{I}$, let $H_\mathcal{I} = C_n$, with $2n$ boundary points in $\mathcal{E}(\mathcal{I})$. With the order $\leq$ and the substance $R$, Jones defines a direct limit $\mathcal{V} = \lim_{\mathcal{I} \in \mathfrak{T}} H_\mathcal{I}$ and a unitary representation of $F$ on $\mathcal{V}$.

To see a concrete example, let $C = (C_n)$ be the planar algebra of Conway tangles, where $C_n$ is the free $\mathbb{Z}$-module generated by Conway tangles with $2n$ boundary points. For Conway tangles, we identify two tangles if they differ by a family of distant unlinked unknots. For example, $C_0$ is spanned by links up to distant unlinked unknots, $C_1$ is spanned by tangles with two boundary points, up to distant unlinked unknots. We denote the class of the trivial tangle in $C_1$ as $\xi$. Moreover, $C_n$ has a $C_0$-valued inner product structure $\langle x, y \rangle$, given by connecting $2n$ boundary points of $x$ to $2n$ boundary points of $y$ one by one (Figure \ref{fig: inner product example}).

\begin{figure}
	\centering
	\includegraphics[width = 0.4\textwidth]{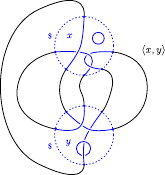}
	\caption{$x, y \in C_3$ are tangles with $6$ boundary points. Their inner product $\langle x, y \rangle \in C_0$ is obtained by connecting boundary points of $x$ to boundary points of $y$ one by one.}
	\label{fig: inner product example}
\end{figure}


Furthermore, let $R \in C_2$ be a crossing. Then take the direct limit over the ordered set $(\mathfrak{T}, \leq)$, with linear map from $H_{\mathcal{K}}$ to $H_{\mathcal{I}}$ induced by concatenating the tangle $T^{\mathcal{K}}_{\mathcal{I}}$ (Figure \ref{fig: forest to unoriented tangle}) along $\mathcal{E}(\mathcal{K})$. The resulting space $\mathcal{V}=\lim_{\mathcal{I} \in \mathfrak{T}} H_{\mathcal{I}}$ has an inner product induced by the inner product on $C_n$.

\begin{figure}
    \centering
    \includegraphics[width = \textwidth]{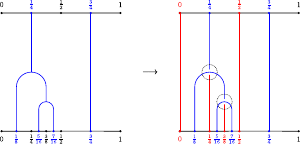}
	\put(-450, 15){$\mathcal{I}$}
	\put(-450, 185){$\mathcal{K}$}
	\put(-450, 120){$\color{blue} \tau^{\mathcal{K}}_{\mathcal{I}}$}
	\put(-195, 15){$\mathcal{I}$}
	\put(-195, 185){$\mathcal{K}$}
	\put(0, 120){$T^{\mathcal{K}}_{\mathcal{I}}$}
	\put(-120, 115){$R$}
	\put(-100, 75){$R$}
    \caption{$\tau^\mathcal{K}_\mathcal{I}$ is an embedded forest representing the refinement from $\mathcal{K}$ to $\mathcal{I}$. Its roots are midpoints of the subintervals of $\mathcal{K}$, and its leaves are midpoints of the subintervals of $\mathcal{I}$. Trivalent vertices lie vertically above the added breakpoints, representing splits of subintervals. Then replace each trivalent vertex of the forest $\tau^{\mathcal{K}}_{\mathcal{I}}$ by a crossing $R\in C_2$, with a vertical straight line joining the bottom of the disc containing $R$ to the breakpoint of $\mathcal{I}$ below the replaced vertex, and vertical lines connecting all the breakpoints of $\mathcal{K}$ to those breakpoints of $\mathcal{I}$ in common with $\mathcal{K}$ except $1$. All the extra strands are marked in red.}
    \label{fig: forest to unoriented tangle}
\end{figure}

Jones constructed a representation $\pi_L$ of $F$ on $\mathcal{V}$ as follows.

\begin{definition}\label{def: pi_L(g)}
    Given an element $g \in F$ and a vector $v\in \mathcal{V}$, we can always find a pair of s.d. partitions $(\mathcal{I}, \mathcal{J})$ representing $g$ such that a representative of $v$ lies in $H_{\mathcal{I}}$. Now define the action of $\pi_L(g)$ on $v$ to be induced by concatenating with the trivial tangle $g^{\mathcal{I}}_{\mathcal{J}}$ (Figure \ref{fig: shift example}) along $\mathcal{E}(\mathcal{I})$.
\end{definition}

As a result, a link class (up to distant unlinked unknots) $\langle \pi_L(g)(\xi), \xi \rangle \in C_0$ arises as a coefficient of $\pi_L(g)(\xi)$, denoted as $L(g)$. See Figure \ref{fig: construction of link with shift example} for an example.

Jones showed that all links can arise this way.

\begin{figure}
    \begin{subfigure}{0.4\textwidth}
        \centering
    	\includegraphics[width =\textwidth]{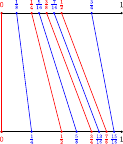}
    	\put(-190, 15){$\mathcal{J}$}
        \put(-190, 180){$\mathcal{I}$}
    	\put(-10, 120){$g^{\mathcal{I}}_{\mathcal{J}}$}
    	\caption{}
    	\label{fig: shift example}
     \end{subfigure}
     \hfill
     \begin{subfigure}{0.45\textwidth}
        \centering
        \includegraphics[width =\textwidth]{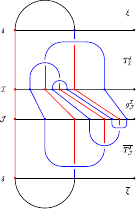}
        \caption{}
        \label{fig: construction of link with shift example}
     \end{subfigure}
     \caption{(A) $g^\mathcal{I}_\mathcal{J}$ is the trivial tangle connecting the points in $\mathcal{E}(\mathcal{I})$ to the points in $\mathcal{E}(\mathcal{J})$ one by one. (B) The canonical element $\xi \in H_{\delta} = C_1$ has a representative $\xi \circ T^\delta_\mathcal{I} \in H_\mathcal{I}$. Then $\pi_L(g)(\xi)$ has a representative $\xi \circ T^\delta_\mathcal{I} \circ g^\mathcal{I}_\mathcal{J} \in H_\mathcal{J}$. After taking the inner product with $\xi$'s another representative $\xi \circ T^\delta_\mathcal{J} \in H_\mathcal{J}$, we obtain the resulting link class $L(g) = \langle \pi_L(g)(\xi), \xi \rangle \in C_0$.}
\end{figure}

\begin{proposition} \textup{\cite[Theorem 5.3.1]{MR3589908}} \label{prop: unoriented Thompson Alexander theorem ver 1}
	For any $[L] \in C_0$, there exists $g\in F$ such that $L(g) = [L]$.
\end{proposition}


Next we consider Temperley-Lieb planar algebra $\mathcal{TL} = (\mathcal{TL}_n)$, where $\mathcal{TL}_n$ is the free $\mathbb{Z}[A, A^{-1}]$-module  generated by all planar tangles
with $2n$ boundary points and with the relation $\bigcirc = -A^2 - A^{-2}$.

\begin{definition}\label{def: Pi_J}
    Jones representation $\pi_J$ is defined similarly to $\pi_L$ in Definition \ref{def: pi_L(g)}, but we consider Temperley-Lieb planar algebra and the following instance $Q \in \mathcal{TL}_2$, which is the Kauffman bracket of $R$.
    \begin{figure}[H]
        \centering
        \includegraphics[width=0.5\linewidth]{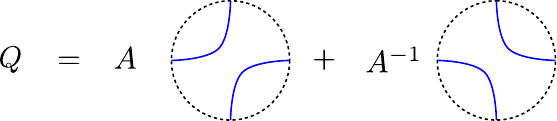}
    \end{figure}
\end{definition}

Furthermore, we multiply the inner product by a normalization factor $\frac{1}{(-A^2-A^{-2})^n}$, so that the normalized inner product takes values in $\mathbb{Q}(A)$, the field of rational functions in $A$.

For $g \in F$, we first take a pair of s.d. $n$-partitions $(\mathcal{I}, \mathcal{J})$ representing it. We denote the link $\xi \circ T^\delta_\mathcal{I} \circ g^\mathcal{I}_\mathcal{J} \circ \overline{T^\delta_\mathcal{J}} \circ \overline{\xi}$ in Figure \ref{fig: construction of link with shift example} as $L(\mathcal{I},\mathcal{J})$, which is a representative of $L(g)$.
Then we have
\begin{equation}\label{equ: normalized Kauffman for Thompson}
    \langle \pi_J(g)(\xi), \xi \rangle = \frac{1}{(-A^2-A^{-2})^n} \langle L(\mathcal{I},\mathcal{J}) \rangle(A),
\end{equation}
the normalized Kauffman bracket of $L(g)$.

Jones also defined a subgroup $\vec{F}$ of $F$ such that for $g \in \vec{F}$, the link $L(g)$ has a canonical orientation. Furthermore, $L(g)$ has writhe $0$, so for $g \in \vec{F}$ we have

\begin{equation}\label{equ: normalized Jones' polynomial for Thompson}
    \langle \pi_J(g)(\xi), \xi \rangle = \frac{1}{(-A^2-A^{-2})^{n-1}}V_{L(\mathcal{I}, \mathcal{J})}(A)
\end{equation}
the normalized Jones polynomial of $L(g)$.

Here we omit the definition of the canonical orientation on $L(g)$. Instead, we will give an alternative construction of the oriented link $L(g)$ in Section \ref{sec:half-grid} using half grid diagrams. 

\section{Equivalent definition of $B_{\vec{F}}$ using concentration structures on categories}

\subsection{Concentration structures and $G$-equivariant direct limits} \label{sec:concentration-structures-on-categories}

In this section we will record some definitions and results from \cite{MR5118516}.

Let $\mathcal{C}$ be a small category. The sets of objects and morphisms of $\mathcal{C}$ are denoted as $\textup{Ob}_{\mathcal{C}}$ and $\textup{Mor}_{\mathcal{C}}$.

\begin{definition} \label{def: concentration structure}
     A \textbf{concentration structure} on $\mathcal{C}$ is an equivalence relation $\sim$ on $\textup{Mor}_{\mathcal{C}}$, which satisfies the following axioms:
    
    \begin{enumerate}
        \item $id_A \sim id_B$ for any objects $A, B \in \textup{Ob}_{\mathcal{C}}$. (Identity axiom)
        
        \item If $f \sim f', g \sim g'$ and both $f\circ g, f' \circ g'$ exist, then $f \circ g \sim f' \circ g'$. (Composition axiom)

        \item For any morphisms $f, g \in \textup{Mor}_{\mathcal{C}}$, there exist some $f' \sim f$ and $ g' \sim g$ such that $f' \circ g'$ exists. (2-Existence axiom)

        \item  For any morphisms $f\sim f', g\sim g', h\sim h', m\sim f\circ g, n \sim g'\circ h'$ we have $f' \circ n \sim m \circ h$ whenever all the compositions are  well defined. (Associativity axiom)
    \end{enumerate}
   
   If $\sim$ is a concentration structure on $\mathcal{C}$, we will say $(\mathcal{C},\sim)$ is a \textbf{category with concentration}.
\end{definition}

\begin{definition} \label{def: concentration preserving functor}
    We say a map $F: (\mathcal{C},\sim_{\mathcal{C}}) \rightarrow(\mathcal{D},\sim_{\mathcal{D}}) $ is a \textbf{concentration preserving functor} if the following holds 
    
    \begin{itemize}
        \item $F$ is a functor between $\mathcal{C}$ and $\mathcal{D}$ in the usual sense.
        \item For any $f,f'\in \textup{Mor}_{\mathcal{C}}$, if $f\sim_{\mathcal{C}}f'$ then $F(f) \sim_{\mathcal{D}}F(f')$. 
    \end{itemize}
\end{definition}

\begin{definition}
    Given $(\mathcal{C},\sim)$, a category with concentration, define the \textbf{concentration monoid} of $(\mathcal{C}, \sim)$ to be $M_{(\mathcal{C}, \sim)} = \textup{Mor}_{\mathcal{C}} / \sim$, with multiplication 
    $$
    [f][g] = [f' \circ g']
    $$
    where $f' \sim f, g' \sim g$ and $f' \circ g'$ exists. The third axiom of concentration structure guarantees that we can always find such $f'$ and $g'$.
\end{definition}

An important property of concentration preserving functors is that they induce monoid homomorphisms on concentration monoids.
\begin{proposition} \cite[Proposition 2.2]{MR5118516}\label{prop: concentration preserving functor induces monoid homomorphism}
    Any concentration preserving functor $F: (\mathcal{C},\sim_{\mathcal{C}}) \rightarrow(\mathcal{D},\sim_{\mathcal{D}})$ induces a monoid homomorphism $\phi_F: M_{(\mathcal{C},\sim_{\mathcal{C}})} \to M_{(\mathcal{D},\sim_{\mathcal{D}})}$ where $\phi_F([f])=[F(f)]$.
\end{proposition}

Recall that a \emph{group action of $G$ on a category $\mathcal{C}$} is a group homomorphism $\rho: G \rightarrow \textup{Aut}(\mathcal{C})$, where $\textup{Aut}(\mathcal{C})$ denotes the group of strongly invertible functors from $\mathcal{C}$ to itself. When the action has no ambiguity, we simply write $ \rho(f)(\_)$ as $f(\_)$ for the action of $f\in G$ on an object or a morphism of $\mathcal{C}$. A functor $F: \mathcal{C} \to \mathcal{D}$ is said to be \emph{$G$-equivariant} if $F(f(C)) = f(F(C))$ for any object $C$ in $\mathcal{C}$, and $F(f(\alpha)) = f(F(\alpha))$ for any morphism $\alpha$ in $\mathcal{C}$.

Let $(S, \leq)$ be a directed set, and let $\mathcal{S}$ be its associated direct category, defined by $\textup{Ob}(\mathcal{S}) = S$ with a single morphism $\iota_A^B$ from $A$ to $B$ whenever $A \leq B$. Suppose a group $G$ acts on $\mathcal{S}$. Note that the morphisms in $\mathcal{S}$ are determined by the partial order on the objects, so the action must satisfy $f(\iota_A^B)=\iota_{f(A)}^{f(B)}$ for any $f\in G$ and any $A \leq B$.


\begin{definition} \label{def: category S_G}
    Given a group action of $G$ on $\mathcal{S}$, and a $G$-equivariant functor $F: \mathcal{S} \to \mathscr{M}on$, the category of monoids (where $G$ acts trivially on $\mathscr{M}on$), define the category $\mathcal{S}_{G}$ as follows:
    \begin{itemize}
        \item $\textup{Ob}(\mathcal{S}_G) = \textup{Ob}(\mathcal{S})$.
        \item $\textup{Mor}(B, A) = \{(A, \alpha, f): \alpha \in F(A), f \in G, f(B) = A\}$.
        \item Composition is given by $(A, \alpha, f) \circ (B, \beta, g)  = (A, \alpha\beta, fg)$. Here $\alpha\beta$ is well defined because $F$ acts $G$-equivariantly on the objects, i.e. $F(A) = F(f(B)) = F(B)$. 
    \end{itemize}
    The identity morphism $id_{A}$ is given by $(A, e_{F(A)}, \epsilon)$, where $e_{F(A)}$ is the identity in $F(A)$ and $\epsilon$ is the identity in $G$. 
\end{definition}

\begin{remark}
    Each morphism $(A, \alpha, f)$ in $\mathcal{S}_G$ implicitly determines its source and target: the target is $A$, while the source is $f^{-1}(A)$. The element $\alpha$ belongs to $F(A)$, the monoid associated with the target object $A$.
\end{remark}

\begin{definition} \label{def: concentration on S_G}
    We define the concentration structure on $\mathcal{S}_G$ such that $(A, \alpha, f) \sim (B, \beta, g)$ if and only if
    \begin{enumerate}
        \item $(F(\iota_A^C))(\alpha) = (F(\iota_B^C))(\beta)$ for some $C \in \textup{Ob}(\mathcal{S})$.
        \item $f = g$.
    \end{enumerate}
\end{definition}

It is shown in \cite{MR5118516} that this is indeed a concentration structure. The $G$-equivariant direct limit is the concentration monoid of $\mathcal{S}_G$ with this concentration structure. More precisely, we have the following definition. 

\begin{definition} \label{def: G-equivariant direct limit}
    Given a direct category $\mathcal{S}$, a $G$-action on $\mathcal{S}$, and a $G$-equivariant functor $F: \mathcal{S} \to \mathscr{M}on$, we define the \textbf{$G$-equivariant direct limit} $\displaystyle \lim_{\longrightarrow}{}^G F$ to be $ M_{(\mathcal{S}_G,\sim)}$, the concentration monoid of $(\mathcal{S}_G, \sim)$.
\end{definition}



\subsection{Maps between $G$-equivariant direct limits} \label{sec:g-equivariant-direct-limits}

In this section, we study maps between $G$-equivariant direct limits, and later we will use them to construct a representation of the dyadic braid group. The definitions and results in this section are not contained in \cite{MR5118516}.

We first introduce $G$-equivariant natural transformations between two $G$-equivariant functors.

\begin{definition}
    Suppose that $\mathcal{C}$ and $\mathcal{D}$ are categories with $G$-actions on them. Let $F_1, F_2$ be two $G$-equivariant functors from $\mathcal{C}$ to $\mathcal{D}$. We say a natural transformation $\tau$ from $F_1$ to $F_2$ is $G$-equivariant if $g(\tau_A) = \tau_{g(A)}$ for any $g \in G$ and any object $A$ in $\mathcal{C}$.
\end{definition}

\begin{proposition} \label{prop:S-functor-induces-SG-functor}
    Let $\mathcal{S}$ and $\mathcal{S}'$ be two direct categories with $G$-actions on them, and let $F, \ F'$ be two $G$-equivariant functors from $\mathcal{S}, \ \mathcal{S}'$ to $\mathcal{M}on$ respectively (recall that we always require that $G$ acts trivially on $\mathcal{M}on$). Then given any $G$-equivariant functor $\rho:\mathcal{S} \rightarrow\mathcal{S}'$, and a $G$-equivariant natural transformation $\tau$ from $F$ to $F' \circ \rho$, $\rho$ and $\tau$ induce a concentration preserving functor
    $$
    \rho_\tau: (\mathcal{S}_G, \sim) \to (\mathcal{S}'_G, \sim')
    $$
    sending object $A$ to $\rho(A)$, morphism $(A, \alpha, f)$ to $(\rho(A), \tau_A(\alpha), f)$.
\end{proposition}

\begin{proof}
    We first check that $\rho_\tau$ is a well-defined functor. Suppose that $(A, \alpha, f)$ is a morphism from $B$ to $A$, then $\rho_\tau(A, \alpha, f) = (\rho(A), \tau_A(\alpha), f)$ is indeed a morphism from $\rho(B)$ to $\rho(A)$, since $\rho$ is $G$-equivariant we have $f(\rho(B)) = \rho(f(B)) = \rho(A)$. It is easy to see that $\rho_\tau$ preserves identity morphisms, since $\tau_A$ is a monoid homomorphism for any $A$. Now given another morphism 
    $(B, \beta, g)$ from $C$ to $B$, we have
    \begin{align*}
        \rho_\tau(A, \alpha, f) \circ \rho_\tau(B, \beta, g) 
        &= (\rho(A), \tau_A(\alpha), f) \circ (\rho(B), \tau_B(\beta), g) \\
        &= (\rho(A), \tau_A(\alpha)\tau_B(\beta), fg) \\
        &= (\rho(A), \tau_A(\alpha)\tau_A(\beta), fg) \\
        &= (\rho(A), \tau_A(\alpha\beta), fg) \\
        &= \rho_\tau(A, \alpha\beta, fg) \\
        &= \rho_\tau((A, \alpha, f) \circ (B, \beta, g))
    \end{align*}
 where from line $2$ to line $3$ we use the fact that $\tau$ is $G$-equivariant, hence $\tau_A(\beta)=\tau_{f(B)}(\beta)=f(\tau_B(\beta))=\tau_B(\beta)$.

    Then we check that $\rho_\tau$ preserves the concentration structures. Suppose that $(A, \alpha, f) \sim (B, \beta, g)$, then $f = g$ and there exists some $C \geq A, B$ such that $(F(\iota_A^C))(\alpha) = (F(\iota_B^C))(\beta)$, and we want to show $(\rho(A),
    \tau_A(\alpha),f) \sim' (\rho(B),
    \tau_B(\beta),f)$. By definition it is enough to show $(F'(\iota_{\rho(A)}^{\rho(C)}))(\tau_A(\alpha))=(F'(\iota_{\rho(B)}^{\rho(C)}))(\tau_B(\beta))$ which follows from the following two commutative diagrams on the naturality.

    \[\begin{tikzcd}
    	{F(C)} & {F'(\rho(C))} &&& {F(C)} & {F'(\rho(C))} \\
    	{F(A)} & {F'(\rho(A))} &&& {F(B)} & {F'(\rho(B))}
    	\arrow["{\tau_C}", from=1-1, to=1-2]
    	\arrow["{\tau_C}", from=1-5, to=1-6]
    	\arrow["{F(\iota_A^C)}", from=2-1, to=1-1]
    	\arrow["{\tau_A}"', from=2-1, to=2-2]
    	\arrow["{F'(\iota_{\rho(A)}^{\rho(C)})}"', from=2-2, to=1-2]
    	\arrow["{F(\iota_B^C)}", from=2-5, to=1-5]
    	\arrow["{\tau_B}"', from=2-5, to=2-6]
    	\arrow["{F'(\iota_{\rho(B)}^{\rho(C)})}"', from=2-6, to=1-6]
    \end{tikzcd}\]

    More precisely, \begin{align*}
(F'(\iota_{\rho(A)}^{\rho(C)}))(\tau_A(\alpha))= \tau_C(F(\iota_A^C)(\alpha))=\tau_C(F(\iota_B^C)(\beta)) =(F'(\iota_{\rho(B)}^{\rho(C)}))(\tau_B(\beta)).
    \end{align*}
\end{proof}

As a direct consequence of Proposition \ref{prop:S-functor-induces-SG-functor} and Proposition \ref{prop: concentration preserving functor induces monoid homomorphism} we have the following corollary.

\begin{corollary} \label{cor:S-functor-induces-equivariant-direct-limit-homomorphism}
    In the setting of Proposition \ref{prop:S-functor-induces-SG-functor}, $\rho$ and $\tau$ further induce a monoid homomorphism $\phi_{\rho_{\tau}}: \displaystyle \lim_{\longrightarrow}{}^G F \to \displaystyle \lim_{\longrightarrow}{}^G F'$.
\end{corollary}

\subsection{$B_{\vec{F}}$ as an $\vec{F}$-equivariant direct limit} \label{sec:BF-as-an-F-equivariant-direct-limit}

Let $S^{\frac{1}{2}}$ be the collection of finite subsets of dyadic rational numbers, ordered by inclusion. Then $(S^{\frac{1}{2}}, \subseteq)$ is a directed set. Denote $\mathcal{S}^{\frac{1}{2}}$ to be the associated direct category. We define the functor $Br:\mathcal{S}^{\frac{1}{2}} \rightarrow \mathcal{G}rp$ as follows:

\begin{itemize}
    \item For any finite set $P$, let $Br(P) = Br_{|P|}$, the braid group with $|P|$ strands.
    \item For any pair $P \subseteq Q$, let $Br(\iota_P^Q): Br_{|P|} \to Br_{|Q|}$ be the group homomorphism induced by $\it{adding \ backside \ trivial \ strands}$ on $Q \setminus P$, where we identify the endpoints of $Br_{|P|}$ and $Br_{|Q|}$ with $P$ and $Q$ respectively. See Figure \ref{fig: adding trivial strands to a thompson braid} for an illustration.
\end{itemize}

\begin{figure}
    \begin{subfigure}{0.4\textwidth}
        \centering
        \includegraphics{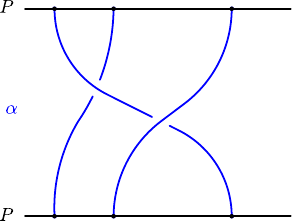}
        \caption{}
        \label{fig: thompson braid before adding trivial strands}
    \end{subfigure}
    \hfill
    \begin{subfigure}{0.4\textwidth}
        \centering
        \includegraphics{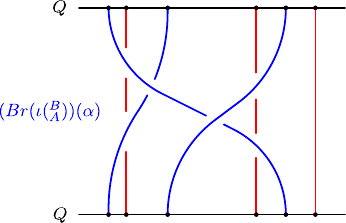}
        \caption{}
        \label{fig: thompson braid after adding trivial strands}
    \end{subfigure}
    \caption{(A) An element $\alpha \in Br_{|P|}$. (B) The group homomorphism $Br(\iota_P^Q)$ adds backside trivial strands (marked in red) to $\alpha$. The endpoints of those trivial strands are on $Q \setminus P$.}
    \label{fig: adding trivial strands to a thompson braid}
\end{figure}

It is easy to check the map defined above is actually a well-defined functor from $\mathcal{S}^{\frac{1}{2}}$ to $\mathcal{G}rp$.

Consider the $\vec{F}$-action $\rho: \vec{F} \to \textup{Aut}(S^{\frac{1}{2}})$ such that $\rho_f(P) = f(P)$ and $\rho_f(\iota_P^Q) = \iota_{f(P)}^{f(Q)}$. Again, we consider the trivial action on $\mathcal{G} rp$.

Since the action of $\vec{F}$ preserves the order of dyadic rational numbers, $Br$ is an $\vec{F}$-equivariant functor. We can take the $\vec{F}$-equivariant direct limit of $Br$.

\begin{proposition}
    The group $B_{\vec{F}}$ and the group $\displaystyle \lim_{\longrightarrow}{}^{\vec{F}} Br$ are isomorphic.
\end{proposition}
\begin{proof}
    By \cite[Proposition 4.6]{MR5118516}, the $\vec{F}$-equivariant direct limit $\displaystyle \lim_{\longrightarrow}{}^{\vec{F}} Br$ is isomorphic to $\displaystyle \lim_{\longrightarrow} Br \rtimes_{\phi} \vec{F}$, where $\displaystyle \lim_{\longrightarrow} Br$ is the regular direct limit, and $\phi: \vec{F} \to \textup{Aut}(\displaystyle \lim_{\longrightarrow} Br)$ is given by $\phi_f[P, \alpha] = [f(P), \alpha]$.

    Let $\sigma$ be the canonical generator of $Br_2$. It is easy to see that the direct limit $\displaystyle \lim_{\longrightarrow} Br$ is generated by the braid generators $\sigma_{ab} = [\{a, b\}, \sigma]$ for any two dyadic rationals $a<b$, with braid relations
    \begin{align*}
      &\sigma_{ab}\sigma_{cd}=\sigma_{cd}\sigma_{ab} \ \text{ if $(a-c)(a-d)(b-c)(b-d)>0$,} \\ 
      &\sigma_{ab}\sigma_{bc}=\sigma_{ac}\sigma_{ab}=\sigma_{bc}\sigma_{ac} \ \text{  for all $a,b,c$ with $a<b<c$}.\\ 
    \end{align*}
    
    The oriented Thompson group $\vec{F}$ is generated by $g_i$ for $i \in \mathbb{Z}_{\geq 0}$ with relations
    $$
    g_n g_l = g_l g_{n+2} \ \text{ for all $l < n$}
    $$

    The semidirect product $\displaystyle \lim_{\longrightarrow} Br \rtimes_{\phi} \vec{F}$ is generated by the union of the generating sets of its two factors. Its defining relations consist of the relations of both factors together with the following additional relations:
    $$
    g_i\sigma_{ab} = (\phi_{g_i}(\sigma_{ab}))g_i \ \text{ for all $i,a \text{ and } b$}
    $$

    Note that $\phi_{g_i}(\sigma_{ab}) = \sigma_{g_i(a) g_i(b)}$, so the generators and relations of $\displaystyle \lim_{\longrightarrow} Br \rtimes_{\phi} \vec{F}$ are the same as the generators and relations of $B_{\vec{F}}$.
\end{proof}

By the above proposition, we can identify $B_{\vec{F}}$ and $\displaystyle \lim_{\longrightarrow}{}^{\vec{F}} Br$ by identifying their generators. Thus we will consider the elements in $B_{\vec{F}}$ as the equivalence classes of morphisms in $\mathcal
{S}^{\frac{1}{2}}_{\vec{F}}$ which are of the form $[P, \alpha, f]$, where $P$ is a finite set of dyadic rational numbers, $\alpha \in Br_{|P|}$ and $f \in \vec{F}$. An element $[P,\alpha,f] \in B_{\vec{F}}$ can be represented by a braid $\alpha$ with bottom endpoints on $P$ and top endpoints on $f^{-1}(P)$. Such a representative is not unique: two diagrams representing the same class $[P,\alpha,f]$ may differ by adding and deleting some $f$-shifted trivial strands, each of which connects a point $p$ to $f(p)$, drawn behind all other strands.

The product of two elements $[P,\alpha,f]$ and $[Q,\beta,g]$ is obtained by concatenating suitable representatives $(P', \alpha', f)$ and $(Q', \beta', g)$ such that $f^{-1}(P') = Q'$. See Figure \ref{fig: example of thompson braid group multiplication} for an example. Note that the braids are drawn from top to bottom (to match the canonical orientations of Thompson links), which is different from the convention in \cite{MR5118516}.

\begin{figure}
    \centering
    \begin{subfigure}[c]{0.4\textwidth}
        \centering
        \includegraphics{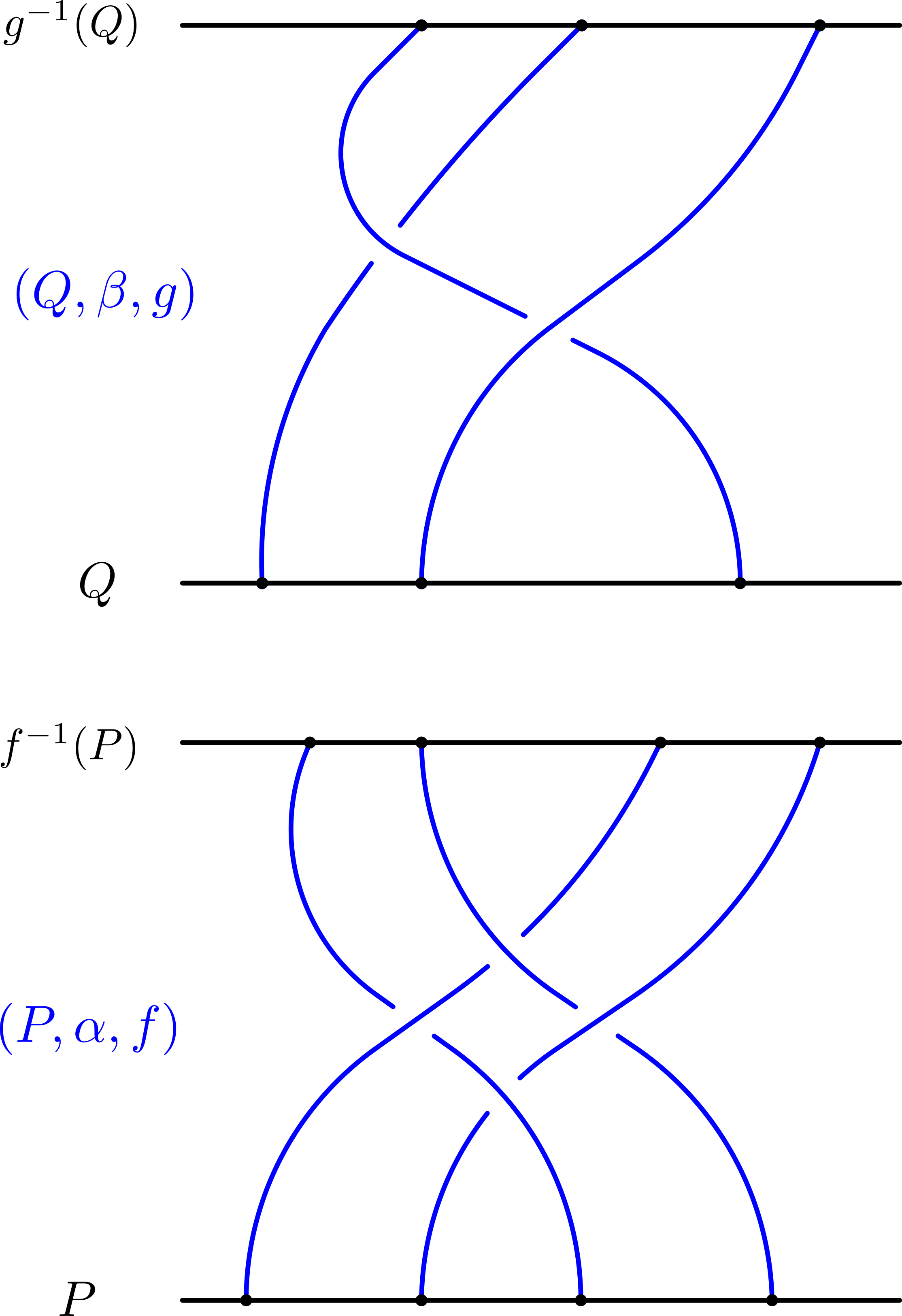}
        \caption{}
        \label{fig: two thompson braids with shifts}
    \end{subfigure}
    \hfill
    \begin{subfigure}[c]{0.45\textwidth}
        \centering
        \includegraphics{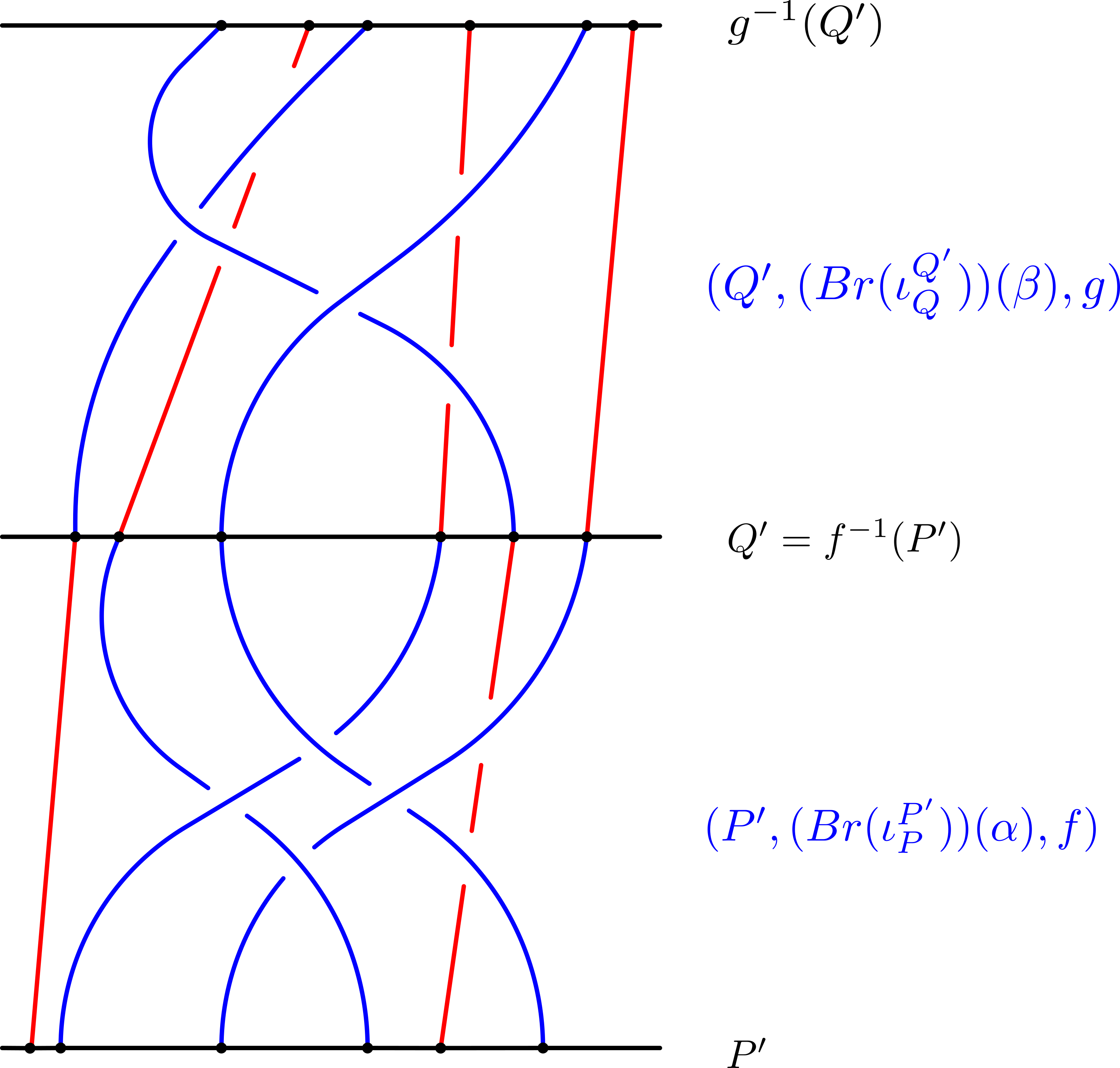}
        \vfill
        \caption{}
        \label{fig: product of thompson braids with shifts}
    \end{subfigure}
    \caption{(A) A diagram of $[P, \alpha, f]$ and a diagram of $[Q, \beta, g]$. (B) To take the product of $[P, \alpha, f]$ and $[Q, \beta, g]$, we add shifted trivial strands so that the two diagrams have matching endpoints in between. Their concatenation is the diagram of the product.}
    \label{fig: example of thompson braid group multiplication}
\end{figure}

Also, according to this isomorphism, it is easy to see that each braid group $Br_n$ can be embedded in $B_{\vec{F}}$. To achieve that, we just need to choose any set $P$ of $n$ points and let $e \in \vec{F}$ be the identity, then $\alpha \mapsto [P, \alpha, e]$ is an inclusion.

Now we can define the closure operation on $B_{\vec{F}}$. 

\begin{definition} \label{def: closed up}
  Define the closure map on $B_{\vec{F}}$ to be $\mathfrak{cl}: B_{\vec{F}}\to C_0$ sending an element $[P,\alpha,f]$ to the closure of the braid $\alpha$ in the usual sense.
\end{definition}

Notice that adding trivial strands in a braid gives distant unlinked unknots in its closure, so $\mathfrak{cl}$ is a well-defined map. By the Alexander theorem of braids, $\mathfrak{cl}$ is also a surjective map.

\section{Jones-Thompson links and the dyadic braid group}

\subsection{Half grid construction of Jones-Thompson links} \label{sec:half-grid}


A useful combinatorial way to encode link diagrams is to use grid diagrams. Each grid diagram $\mathbb{G}$ can give rise to a link diagram $\mathcal{L}(\mathbb{G})$. We assume readers are familiar with grid diagrams and refer the readers to \cite{luo2024halfgriddiagramsthompson} for a detailed discussion. In \cite{luo2024halfgriddiagramsthompson}
two authors defined half grid diagrams and associated tangles.

\begin{definition} \label{def: half grid diagram}
	A half grid diagram $\mathbb{H}$ is an $n \times 2n$ grid of squares, each of which is either empty or contains an $X$ or an $O$. We require that in each row there are exactly two nonempty squares, one containing an $X$ and the other containing an $O$, and each column has exactly one nonempty square (which may contain either an $X$ or an $O$). 
\end{definition}

\begin{definition}
	Let $\mathbb{H}$ be an $n \times 2n$ half grid diagram. The $(0,2n)$-oriented tangle $\mathcal{T}(\mathbb{H})$ associated to $\mathbb{H}$ is defined by first connecting $X$ to $O$ in each row, then connecting each $X$ or $O$ to the bottom of the grid by vertical arcs, oriented upward in the case of an $X$ and downward in the case of an $O$. We again use the crossing convention that horizontal segments are always above vertical segments.
\end{definition}

We say that $\mathbb{H}_+$ and $\mathbb{H}_-$ are compatible if the $i^{th}$ column of $\mathbb{H}_+$ and the $i^{th}$ column of $\mathbb{H}_-$ contain the same symbol for each $i$.

Given two compatible $n \times 2n$ half grid diagrams $\mathbb{H}_+$ and $\mathbb{H}_-$, we can obtain a $2n \times 2n$ grid diagram as follows: We first flip $\mathbb{H}_-$ upside down and change each $X$ to $O$ and $O$ to $X$ to get $-\overline{\mathbb{H}_-}$, then we stack $\mathbb{H}_+$ over $-\overline{\mathbb{H}_-}$ (see Figure \ref{fig: compatible half grid pair example} for an example). The resulting diagram, still denoted as $(\mathbb{H}_+, \mathbb{H}_-)$, is a $2n \times 2n$ grid diagram, because the $i^{th}$ column of $\mathbb{H}_+$ and the $i^{th}$ column of $-\overline{\mathbb{H}_-}$ contain opposite markings for each $i$. It is easy to observe that $\mathcal{L}(\mathbb{H}_+, \mathbb{H}_-) = (-\overline{\mathcal{T}({\mathbb{H}_-})}) \mathcal{T}(\mathbb{H}_+)$, where $-\overline{\mathcal{T}({\mathbb{H}_-})}$ is the vertical mirror of $\mathcal{T}({\mathbb{H}_-})$ with reversed orientation.

\begin{figure} 
	\centering
	\includegraphics[width = 0.5\textwidth]{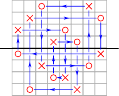}
    \put(0, 40){$-\overline{\mathbb{H}_-}$}
    \put(0, 130){$\mathbb{H}_+$}
	\caption{A pair of compatible half grid diagrams $\mathbb{H}_+$ and $\mathbb{H}_-$ can reproduce a grid diagram $(\mathbb{H}_+, \mathbb{H}_-)$.}
	\label{fig: compatible half grid pair example}
\end{figure}

Throughout this section, we will view a grid or a half grid diagram as a subset of lattice points $\mathbb{Z}_+^2$ where the grid at the lower left corner corresponds to point $(1,1)$. To describe a grid diagram or half grid diagram we just need to specify the coordinates of $X$'s and $O$'s on lattice points.

\vspace{5mm}

Let $\mathcal{E}$ be the set of all dyadic numbers on $[0,1)$. Each number $a \in \mathcal{E}$ can be written in binary. We define the conjugate of $a$ by switching the second-to-last digit of $a$, denoted as $\overline{a}$. For example, $\frac{5}{8} = 0.101$, and its conjugate is $\frac{7}{8} = 0.111$. As a convention, we let $0$ and $\frac{1}{2}$ be conjugate to each other.

As described in \cite{MR3565428} \cite{MR4886157}, each dyadic number $a \in \mathcal{E}$ is assigned a weight $\omega(a) \in \mathbb{Z}/2$ as the mod $2$ sum of the binary digits of $a$. For example, $\omega(\frac{5}{8}) = 1 + 0 + 1 \pmod 2 = 0$. We let $0$ have sign $-$ as a convention, then the weight function gives each point in $\mathcal{E}$ a sign.

Let $\mathcal{I}$ be a s.d. partition, recall that $\mathcal{E}(\mathcal{I})$ is the union of endpoints and midpoints of $\mathcal{I}$ except $1$.

\begin{definition}
    Define $\mathcal{E}_+(\mathcal{I})$ to be the subset of points in $\mathcal{E}(\mathcal{I})$ with $+$ sign, define $\mathcal{E}_-(\mathcal{I})$ to be the subset of points in $\mathcal{E}(\mathcal{I})$ with $-$ sign.
\end{definition}

We consider two orders $\leq_x$ and $\leq_y$ on $\mathcal{E}$ as follows. The order $\leq_x$ is the normal order of real numbers. Define $a \leq_y b$ if $a$'s digit length in binary is strictly larger than $b$'s, or they have the same length but $a \leq_x b$. For example, $0.01 \leq_y 0.11 \leq_y 0.1$. It is easy to check that $\leq_y$ is a total order.




We first restrict $\leq_x$ to $\mathcal{E}(\mathcal{I})$, then order the elements in $\mathcal{E}(\mathcal{I})$ as $x_1 <_x x_2 <_x \cdots <_x x_{2n}$. We define a bijection

\begin{align*}
	p_x: \mathcal{E}(\mathcal{I})  &\to \{1,2,...,2n\} \\
	x_i &\mapsto i, \quad i = 1,2,...,2n
\end{align*}

Then restrict $\leq_y$ to $\mathcal{E}_+(\mathcal{I})$ and reorder elements in $\mathcal{E}_+(I)$ as $y_1 <_y y_2 <_y ... <_y y_n$. We define

\begin{align*}
	p_y: & \mathcal{E}(\mathcal{I}) \to \{1,2,...,n\} \\
	&\begin{cases}
		y_i \mapsto i, & i=1,2,...,n\\
		 \overline{y_i} \mapsto i, & i=1,2,...,n
	\end{cases}
\end{align*}

Then $p_y|_{\mathcal{E}_+(\mathcal{I})}: \mathcal{E}_+(\mathcal{I}) \to \{1,2,...,n\}$ and $p_y|_{\mathcal{E}_-(\mathcal{I})}: \mathcal{E}_-(\mathcal{I}) \to \{1,2,...,n\}$ are both bijections.

Next, we will show how to construct an $n \times 2n$ half grid diagram associated to $\mathcal{I}$. We let each point $a \in \mathcal{E}(\mathcal{I})$ assign a marking to the $(p_x(a), p_y(a))^{th}$ entry in the  $n \times 2n$ grid, where the marking is $X$ if $a \in \mathcal{E}_+(\mathcal{I})$, $O$ if $a \in \mathcal{E}_-(\mathcal{I})$. We denote the resulting diagram as $\mathbb{H}_\mathcal{I}$.

\begin{proposition}\cite[Theorem 4.10]{luo2024halfgriddiagramsthompson} {\label{prop: half grid diagram from SDI}}
	$\mathbb{H}_{\mathcal{I}}$ is a half grid diagram.
\end{proposition}

See Figure \ref{fig: corner isotopy} and \ref{fig: tangle to half grid} for an illustration of the construction of $\mathbb{H}_{\mathcal{I}}$. 
The equivalence between the half grid construction and Jones' original construction of Thompson links was established in the following theorem. Compare Figure \ref{fig:half-grid-pair-with-shift} and Figure \ref{fig: construction of link with shift example} for a schematic idea.

\begin{theorem} \cite[]{luo2024halfgriddiagramsthompson}\label{thm: equivalent oriented links}
	Suppose that $(\mathcal{I}, \mathcal{J})$ is a pair of s.d. partitions representing $g \in \vec{F}$, then
	$$
	\mathcal{L}(\mathbb{H}_{\mathcal{I}}, \mathbb{H}_{\mathcal{J}})= L(\mathcal{I}, \mathcal{J})
	$$
	as oriented links.
\end{theorem}

\begin{figure}
    \centering
    \begin{subfigure}{0.4\textwidth}
        \centering
        \includegraphics[width=\textwidth]{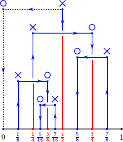}
        \put(-205, 15){$\mathcal{E}(\mathcal{I})$}
        \caption{}
        \label{fig: corner isotopy}
    \end{subfigure}
    \hfill
    \begin{subfigure}{0.4\textwidth}
        \centering
        \includegraphics[width=\textwidth]{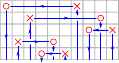}
        \put(-200, 50){$\mathbb{H}_{\mathcal{I}}$}
        \caption{}
        \label{fig: tangle to half grid}
    \end{subfigure}
    \caption{(A) An illustration of the construction of $\mathbb{H}_\mathcal{I}$. Each point in $\mathcal{E}(\mathcal{I})$ corresponds to a marking vertically above it, whose height is determined by the length of its binary expansion. Note that it is a ``cornered isotopy'' of $\xi \circ T^\delta_\mathcal{I}$ in Figure \ref{fig: construction of link with shift example} (with a canonical orientation). (B) $\mathbb{H}_\mathcal{I}$ as a standard half grid diagram.}
\end{figure}


\subsection{From Thompson's group to the dyadic braid group}

In this section we will define the map $\iota$ in Theorem \ref{thm: embedding-preserves-link}, and verify that $\iota$ is a group embedding.

\begin{definition}
    Given an element $g\in \vec{F}$, we define $\iota(g)=[P,\alpha,g]\in B_{\vec{F}}$, where $P, \alpha $ are defined through the following procedures:

    \begin{enumerate}
        \item Find any pair of s.d. partitions $(\mathcal{I}, \mathcal{J})$ representing $g$.
        \item Let $P$ be $\mathcal{E}_+(\mathcal{J})$, the collection of endpoints and midpoints of $\mathcal{J}$ with $+$ sign.
        \item Let $\alpha$ be $\mathfrak{b}(\mathbb{H}_{\mathcal{I}}, \mathbb{H}_{\mathcal{J}})$, the braid induced from the grid diagram $(\mathbb{H}_{\mathcal{I}}, \mathbb{H}_{\mathcal{J}})$ as follows:
    \end{enumerate}

Given a grid diagram $\mathbb{G}$, one obtains a link diagram $\mathcal{L}(\mathbb{G})$ by drawing horizontal and vertical lines joining $X$'s and $O$'s. We turn it into a braid $\mathfrak{b}(\mathbb{G})$ by replacing any vertical line oriented upwards from $O$ to $X$ by two vertical lines, one pointing downwards from the $O$, one pointing downwards to the $X$, and having these new vertical lines pass under all horizontal line segments. One can check that the closure of the braid $\mathfrak{b}(\mathbb{G})$ is exactly the link $\mathcal{L}(\mathbb{G})$. This grid-to-braid operation is also used in \cite{MR2500576}\cite{MR1339757}\cite{MR2232855}, see Figure \ref{fig:grid-to-braid-example} for a specific example.

\begin{figure}
	\centering
	\includegraphics[width = 0.6\textwidth]{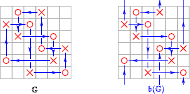}
	\caption{An example of turning a grid diagram $\mathbb{G}$ into a braid $\mathfrak{b}(\mathbb{G})$.}
	\label{fig:grid-to-braid-example}
\end{figure}

\end{definition}

Overall, the dyadic braid $[P, \alpha, g]$ is obtained by doing the above grid-to-braid operation on the grid diagram $(\mathbb{H}_{\mathcal{I}}, \mathbb{H}_\mathcal{J})$ with a $g$-shift between the two half grids (Figure \ref{fig:half-grid-pair-with-shift}). The resulting dyadic braid $[P, \alpha, g]$ is shown in Figure \ref{fig:half-grid-braid-with-shift}.

\begin{figure}\label{fig: dyadic grid to dyadic braid}
    \centering
    \begin{subfigure}{0.4\textwidth}
        \includegraphics[width=\textwidth]{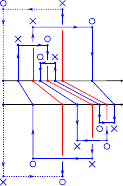}
        \put(-205, 145){$\mathcal{E}(\mathcal{I})$}
        \put(-205, 110){$\mathcal{E}(\mathcal{J})$}
        \caption{}
        \label{fig:half-grid-pair-with-shift}
    \end{subfigure}
    \hfill
    \begin{subfigure}{0.4\textwidth}
        \includegraphics[width=\textwidth]{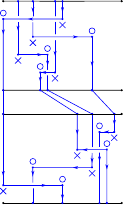}
        \put(-210, 282){$\mathcal{E}_+(\mathcal{I})$}
        \put(-210, 0){$\mathcal{E}_+(\mathcal{J})$}
        \put(-210, 157){$\mathcal{E}_-(\mathcal{I})$}
        \put(-210, 122){$\mathcal{E}_-(\mathcal{J})$}
        \caption{}
        \label{fig:half-grid-braid-with-shift}
    \end{subfigure}
    \caption{(A) Suppose $g \in \vec{F}$ is represented by $(\mathcal{I}, \mathcal{J})$. Recall that Figure \ref{fig: corner isotopy} is obtained by doing a ``cornered'' isotopy to $\xi \circ T^\delta_\mathcal{I}$ in Figure \ref{fig: construction of link with shift example}. We do a similar isotopy to $\overline{T^\delta_\mathcal{J}} \circ \overline{\xi}$. The convention for $X$ and $O$ markings is given in Section \ref{sec:half-grid}. Then we obtain a grid diagram $(\mathbb{H}_\mathcal{I}, \mathbb{H}_\mathcal{J})$ with shift. (B) The dyadic braid $\iota(g)=[P, \alpha, g]$ is obtained by doing the grid-to-braid operation on the grid diagram $(\mathbb{H}_{\mathcal{I}}, \mathbb{H}_\mathcal{J})$ with shift in (a).}
\end{figure}

\begin{lemma}\label{lem: iota is well defined}
    $\iota: \vec{F} \to B_{\vec{F}}$ is a well-defined map.
\end{lemma}

\begin{proof}
    Given an element $g \in \vec{F}$, suppose that $(\mathcal{I}, \mathcal{J})$ and $(\mathcal{I}', \mathcal{J}')$ are two pairs of s.d. partitions representing $g$, we want to show that $[\mathcal{E}_+(\mathcal{J}), \mathfrak{b}(\mathbb{H}_{\mathcal{I}}, \mathbb{H}_{\mathcal{J}}), g] = [\mathcal{E}_+(\mathcal{J}'), \mathfrak{b}(\mathbb{H}_{\mathcal{I}'}, \mathbb{H}_{\mathcal{J}'}), g]$.
    
    Let $(\mathcal{I}'', \mathcal{J}'')$ be a common refinement of $(\mathcal{I}, \mathcal{J})$ and $(\mathcal{I}', \mathcal{J}')$. Note that the grid diagram $(\mathbb{H}_{\mathcal{I}''}, \mathbb{H}_{\mathcal{J}'')}$ differs from $(\mathbb{H}_\mathcal{I}, \mathbb{H}_\mathcal{J})$ by adding some distant unlinked unknots lying above all other strands (Figure \ref{fig:adding-trivial-unknots-to-half-grids}). Hence after we turn $(\mathbb{H}_{\mathcal{I}''}, \mathbb{H}_{\mathcal{J}''})$ into a braid $\mathfrak{b}(\mathbb{H}_{\mathcal{I}''}, \mathbb{H}_{\mathcal{J}''})$, with some Reidemeister II moves, those distant unknots become trivial strands lying under all other strands (Figure \ref{fig:adding-trivial-strand-to-half-grids}). The same holds for $(\mathbb{H}_{\mathcal{I}'}, \mathbb{H}_{\mathcal{J}'})$ and $(\mathbb{H}_{\mathcal{I}''}, \mathbb{H}_{\mathcal{J}''})$. Moreover, by the procedure of turning the grid diagram into a braid (depicted in Figure \ref{fig:half-grid-pair-with-shift} and Figure \ref{fig:half-grid-braid-with-shift}) we see the refinement of the dyadic interval exactly corresponds to adding trivial strands on the dyadic braids. Thus, $[\mathcal{E}_+(\mathcal{J}), \mathfrak{b}(\mathbb{H}_{\mathcal{I}}, \mathbb{H}_{\mathcal{J}}), g] = [\mathcal{E}_+(\mathcal{J}'), \mathfrak{b}(\mathbb{H}_{\mathcal{I}'}, \mathbb{H}_{\mathcal{J}'}), g]=[\mathcal{E}_+(\mathcal{J}''), \mathfrak{b}(\mathbb{H}_{\mathcal{I}''}, \mathbb{H}_{\mathcal{J}''}), g]$.

\end{proof}

\begin{figure}
    \centering
    \includegraphics[width=\linewidth]{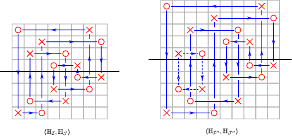}
    \caption{If $(\mathcal{I}'', \mathcal{J}'')$ is a refinement of $(\mathcal{I}, \mathcal{J})$, then $(\mathbb{H}_\mathcal{I''}, \mathbb{H}_\mathcal{J''})$ differs from $(\mathbb{H}_\mathcal{I}, \mathbb{H}_\mathcal{J})$ by adding some distant unlinked unknots lying above all other strands.}
    \label{fig:adding-trivial-unknots-to-half-grids}
\end{figure}

\begin{figure}
    \centering
    \includegraphics[width=\linewidth]{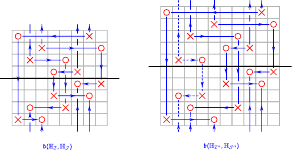}
    \caption{Turning the two grid diagrams in Figure \ref{fig:adding-trivial-unknots-to-half-grids} into braids, we can see that the added distant unlinked unknots become trivial strands lying under all other strands in the braid.}
    \label{fig:adding-trivial-strand-to-half-grids}
\end{figure}

\begin{lemma} \label{lem: iota preserves link}
    The map $\iota: \vec{F} \to B_{\vec{F}}$ satisfies $L(g)=\mathfrak{cl}(\iota(g))$ for all $g \in \vec{F}$. That is, $\iota$ preserves the link.
\end{lemma}

\begin{proof}
    For any element $g \in \vec{F}$, let $(\mathcal{I}, \mathcal{J})$ be a pair of s.d. partitions representing $g$. Then
    $$
    \mathfrak{cl}(\iota(g)) = \mathfrak{cl}[\mathcal{E}_+(\mathcal{J}), \mathfrak{b}(\mathbb{H}_{\mathcal{I}}, \mathbb{H}_{\mathcal{J}}), g] = [\textup{cl}(\mathfrak{b}(\mathbb{H}_{\mathcal{I}}, \mathbb{H}_{\mathcal{J}}))] = [L(\mathbb{H}_{\mathcal{I}}, \mathbb{H}_{\mathcal{J}})]
    $$

    In the equation above we take classes of links in $C_0$. By Theorem \ref{thm: equivalent oriented links}, we have $[L(\mathbb{H}_{\mathcal{I}}, \mathbb{H}_{\mathcal{J}})] = [L(\mathcal{I},\mathcal{J})]=L(g)$. Thus $L(g)=\mathfrak{cl}(\iota(g))$.
\end{proof}

\begin{theorem} \label{thm: iota injective homomorphism}
    $\iota: \vec{F} \to B_{\vec{F}}$ is an injective group homomorphism.
\end{theorem}

\begin{proof}
    Suppose that $\iota(g) = \iota(h)$, then $g = h$ by the definition of $\iota$ and Definition \ref{def: concentration on S_G}. Hence $\iota$ is injective.

    Next we prove that $\iota$ preserves the multiplication. Given $f, g \in \vec{F}$, to take the product of $g$ and $f$, we choose a s.d. partition pair $(\mathcal{I}, \mathcal{J})$ representing $f$, and a s.d. partition pair $(\mathcal{J}, \mathcal{K})$ representing $g$. Then $gf$ is represented by the s.d. partition pair $(\mathcal{I}, \mathcal{K})$. To show that $\iota(g) \iota(f) = \iota(gf)$, we just need to check that
    $$
    [\mathcal{E}_+(K), \mathfrak{b}(\mathbb{H}_{\mathcal{J}}, \mathbb{H}_{\mathcal{K}}), g][\mathcal{E}_+(J), \mathfrak{b}(\mathbb{H}_{\mathcal{I}}, \mathbb{H}_{\mathcal{J}}), f] = [\mathcal{E}_+(K), \mathfrak{b}(\mathbb{H}_{\mathcal{I}}, \mathbb{H}_{\mathcal{K}}), gf]
    $$

    Note that $(\mathcal{E}_+(J), \mathfrak{b}(\mathbb{H}_{\mathcal{I}}, \mathbb{H}_{\mathcal{J}}), f)$ is a morphism from $\mathcal{E}_+(\mathcal{I})$ to $\mathcal{E}_+(\mathcal{J})$, and $(\mathcal{E}_+(K), \mathfrak{b}(\mathbb{H}_{\mathcal{J}}, \mathbb{H}_{\mathcal{K}}), g)$ is a morphism from $\mathcal{E}_+(\mathcal{J})$ to $\mathcal{E}_+(\mathcal{K})$, we have
    $$
    [\mathcal{E}_+(K), \mathfrak{b}(\mathbb{H}_{\mathcal{J}}, \mathbb{H}_{\mathcal{K}}), g][\mathcal{E}_+(J), \mathfrak{b}(\mathbb{H}_{\mathcal{I}}, \mathbb{H}_{\mathcal{J}}), f] = [\mathcal{E}_+(K), \mathfrak{b}(\mathbb{H}_{\mathcal{J}}, \mathbb{H}_{\mathcal{K}}) \mathfrak{b}(\mathbb{H}_{\mathcal{I}}, \mathbb{H}_{\mathcal{J}}), gf]
    $$
    
    It remains to check that $\mathfrak{b}(\mathbb{H}_{\mathcal{J}}, \mathbb{H}_{\mathcal{K}}) \mathfrak{b}(\mathbb{H}_{\mathcal{I}}, \mathbb{H}_{\mathcal{J}}) = \mathfrak{b}(\mathbb{H}_{\mathcal{I}}, \mathbb{H}_{\mathcal{K}})$. Indeed, after we turn $(\mathbb{H}_{\mathcal{J}}, \mathbb{H}_{\mathcal{K}})$ and $(\mathbb{H}_{\mathcal{I}}, \mathbb{H}_{\mathcal{J}})$ into the braids and concatenate them, the lower half of $\mathfrak{b}(\mathbb{H}_{\mathcal{I}}, \mathbb{H}_{\mathcal{J}})$ and the upper half of $\mathfrak{b}(\mathbb{H}_{\mathcal{J}}, \mathbb{H}_{\mathcal{K}})$ become mirrors to each other, and canceled by Reidemeister II moves (see Figure \ref{fig:half-grids-multiplication} for an example). Thus $\mathfrak{b}(\mathbb{H}_{\mathcal{J}}, \mathbb{H}_{\mathcal{K}}) \mathfrak{b}(\mathbb{H}_{\mathcal{I}}, \mathbb{H}_{\mathcal{J}}) = \mathfrak{b}(\mathbb{H}_{\mathcal{I}}, \mathbb{H}_{\mathcal{K}})$.
\end{proof}

\begin{figure}
    \centering
    \includegraphics[width=\linewidth]{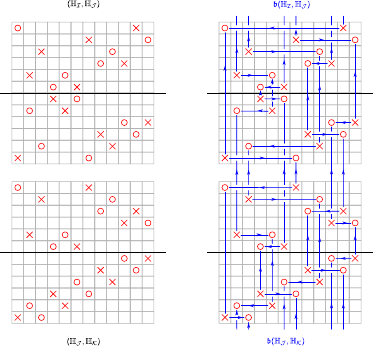}
    \caption{When we take the product of $\mathfrak{b}(\mathbb{H}_\mathcal{J}, \mathbb{H}_\mathcal{K})$ and $\mathfrak{b}(\mathbb{H}_\mathcal{I}, \mathbb{H}_\mathcal{J})$, the lower half of $\mathfrak{b}(\mathbb{H}_\mathcal{I}, \mathbb{H}_\mathcal{J})$ and the upper half of $\mathfrak{b}(\mathbb{H}_\mathcal{J}, \mathbb{H}_\mathcal{K})$ become mirrors to each other. Then they can be canceled out by some Reidemeister II moves.}
    \label{fig:half-grids-multiplication}
\end{figure}

Theorem \ref{thm: embedding-preserves-link} follows by combining Lemma \ref{lem: iota is well defined}, \ref{lem: iota preserves link} and Theorem \ref{thm: iota injective homomorphism}.

\section{Unifying the two representations}

\subsection{Extending braid representations to dyadic braids} \label{sec:extending-braid-representations-to-dyadic-braids}

Recall that $S^{\frac{1}{2}}$ is the directed set of finite subsets of dyadic rational numbers, with partial order subset inclusion $\subseteq$. Let $\mathcal{S}^{\frac{1}{2}}$ be its associated direct category. Now define a functor $TL: \mathcal{S}^{\frac{1}{2}} \to \mathcal{M}on$ as follows.

\begin{itemize}
    \item For any object $P$, let $TL(P) = TL_{|P|}$, the Temperley-Lieb algebra with $|P|-1$ canonical generators $e_1, e_2, ..., e_{|P|-1}$.
    
    \item For any pair $P \subseteq Q$, defining the map $TL(\iota_P^Q): TL_{|P|} \to TL_{|Q|}$ requires more care. We first define it on the generators when $|Q|=|P|+1$.
\end{itemize}

Recall the generator $e_i \in TL_n$ can be described as an $(n,n)$-planar tangle with an upper turnback and a lower turnback both between the $i^{th}$ point and the $(i+1)^{th}$ point, and trivial strands elsewhere. We will view the generator $e_i$ in $TL(P)$ as this diagram but with upper endpoints and lower endpoints both on $P$.

Suppose that $Q$ is obtained by adding a single point between the $k^{th}$ and the $(k+1)^{th}$ points of $P$. We denote the generators of $TL_{|P|}$ by $e_1, e_2, ..., e_{|P|-1}$, and denote the generators of $TL_{|P|+1}$ by $e'_1, e'_2, ..., e'_{|P|}$. Then define $TL(\iota_P^Q): TL_{|P|} \to TL_{|P|+1}$ on the generators as follows.

$$
TL(\iota_P^Q)(e_i)=
\begin{cases}
    e'_i &  i < k \\
    A^2 e'_i e'_{i+1} + e'_i + e'_{i+1} + A^{-2}e'_{i+1} e'_i & i = k\\
    e'_{i+1} & i > k
\end{cases}
$$

Alternatively, $TL(\iota_P^Q)(e_i)$ is the Kauffman bracket of the tangle obtained by adding a backside trivial strand to $e_i$ between the $k^{th}$ point and the $(k+1)^{th}$ point. When $i<k$ or $i>k$, adding a trivial strand doesn't introduce any crossings. When $i = k$, adding a trivial strand will introduce two crossings. Then we use the Kauffman relation to resolve these two crossings as Figure \ref{fig:kauffman-bracket-of-adding-trivial-strand} shows.

\begin{figure}
    \centering
    \includegraphics[width=\linewidth]{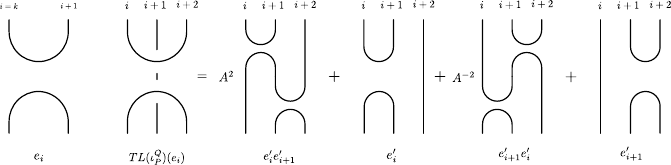}
    \caption{When $i=k$, adding a trivial strand to $e_i$ introduces two crossings. We define $TL(\iota_P^Q)(e_i)$ to be the Kauffman bracket of the resulting tangle.}
    \label{fig:kauffman-bracket-of-adding-trivial-strand}
\end{figure}

Then we extend $TL(\iota_P^Q)$ to $TL_{|P|}$ as an algebra homomorphism.

\begin{lemma}
     $TL(\iota_P^Q): TL_{|P|} \to TL_{|P|+1}$ is a well-defined map.
\end{lemma}

\begin{proof}
    We just need to check that $TL(\iota_P^Q)(e_i)$ preserves the relations on $TL_{|P|}$ and $TL_{|P|+1}$. More precisely, we need to check
    \begin{enumerate}
        \item $TL(\iota_P^Q)(e_i^2) = (-A^2-A^{-2})TL(\iota_P^Q)(e_i)$.
        \item $TL(\iota_P^Q)(e_i e_{i+1} e_i) = TL(\iota_P^Q)(e_i)$.
        \item $TL(\iota_P^Q)(e_i e_{i-1} e_i)= TL(\iota_P^Q)(e_i)$.
        \item $TL(\iota_P^Q)(e_i e_j)  = TL(\iota_P^Q)(e_j e_i)$ for $|i-j| \geq 2$.
    \end{enumerate}
    
    If $TL(\iota_P^Q)$ doesn't introduce crossings, the above equations clearly hold. Otherwise, they follow from the Kauffman bracket interpretation of $TL(\iota_P^Q)$, combined with the fact that the Kauffman bracket is an invariant under regular isotopy and Reidemeister II moves (see Figure \ref{fig:relations-with-trivial-strand}).

    \begin{figure}
        \begin{subfigure}{0.35\textwidth}
            \centering
            \includegraphics[width=\textwidth]{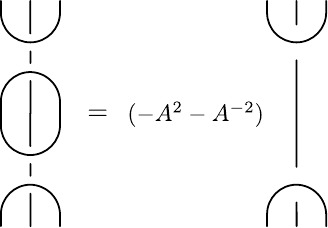}
            \caption{}
            \label{fig:relation-1}
        \end{subfigure}
        \vspace{0.5cm}
        \vfill
        \begin{subfigure}{0.3\textwidth}
            \centering
            \includegraphics[width=\textwidth]{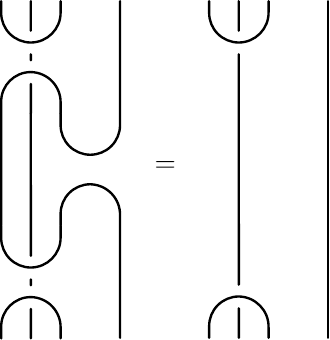}
            \caption{}
            \label{fig:relation-2-1}
        \end{subfigure}
        \hspace{3cm}
        \begin{subfigure}{0.3\textwidth}
            \centering
            \includegraphics[width=\textwidth]{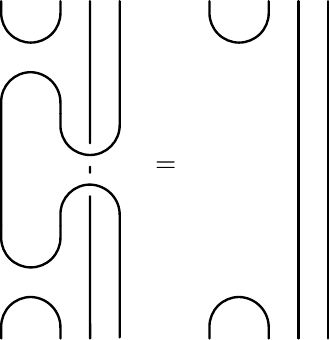}
            \caption{}
            \label{fig:relation-2-2}
        \end{subfigure}
        \vspace{0.5cm}
        \vfill
        \begin{subfigure}{0.3\textwidth}
            \centering
            \includegraphics[width=\textwidth]{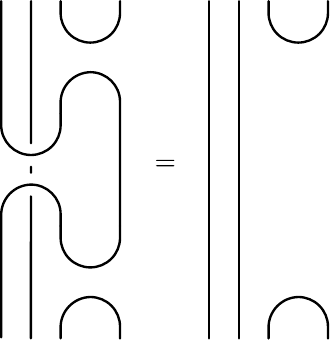}
            \caption{}
            \label{fig:relation-3-1}
        \end{subfigure}
        \hspace{3cm}
        \begin{subfigure}{0.3\textwidth}
            \centering
            \includegraphics[width=\textwidth]{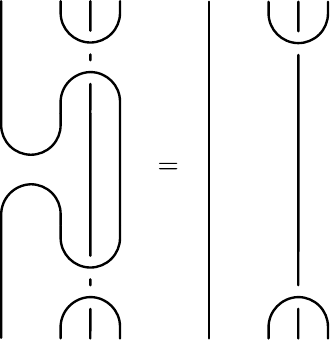}
            \caption{}
            \label{fig:relation-3-2}
        \end{subfigure}
        \vspace{0.5cm}
        \vfill
        \begin{subfigure}{0.4\textwidth}
            \centering
            \includegraphics[width=\textwidth]{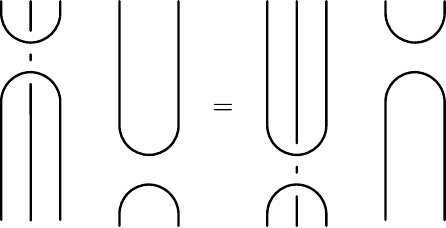}
            \caption{}
            \label{fig:relation-4-1}
        \end{subfigure}
        \hfill
        \begin{subfigure}{0.4\textwidth}
            \centering
            \includegraphics[width=\textwidth]{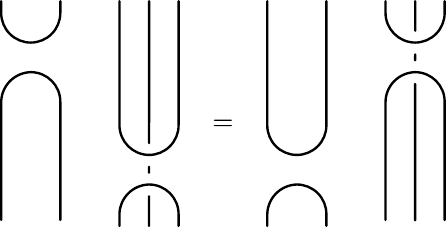}
            \caption{}
            \label{fig:relation-4-2}
        \end{subfigure}
        \caption{(A) illustrates the first relation with $k=i$. (B) and (C) correspond to the two cases of the second relation: $k=i$ or $k=i+1$. (D) and (E) correspond to the two cases of the third relation: $k=i-1$ or $k=i$. (F) and (G) correspond to the two cases of the fourth relation: $k=i$ or $k=j$.}
        \label{fig:relations-with-trivial-strand}
    \end{figure}
\end{proof}

In general, suppose that $Q$ is obtained by adding $n$ points to $P$. Then we consider a sequence $P=P_0, P_1, ..., P_n=Q$ such that $|P_{i+1}| = |P_i| + 1$. Define
$$
TL(\iota_P^Q) = TL(\iota_{P_{n-1}}^Q)  \circ  \cdots \circ TL(\iota_{P_1}^{P_2}) \circ TL(\iota_{P}^{P_1}) 
$$

Note that such a composition means adding $n$ trivial strands, then doing Kauffman brackets in the order of the given sequence. It doesn't depend on the choice of the sequence, so $TL(\iota_P^Q)$ is well defined. The same argument shows that $TL$ is indeed a functor.

Since the elements in $\vec{F}$ are order-preserving, the functor $TL$ is $\vec{F}$-equivariant. Then we can take its $\vec{F}$-equivariant direct limit.

\begin{definition}
    Define the oriented dyadic Temperley-Lieb monoid $\overline{TL_{\vec{F}}}$ to be $\displaystyle \lim_{\longrightarrow}{}^{\vec{F}} TL$, the $\vec{F}$-equivariant direct limit of $TL$.
\end{definition}

\begin{definition}
    Define the oriented dyadic Temperley-Lieb algebra $TL_{\vec{F}}$ to be the free $\mathbb{Z}$-module generated by the elements in $\overline{TL_{\vec{F}}}$, with the following linearity relation:
    $$
    a[P, \alpha, f] + b[P, \beta, f] \sim [P, a\alpha + b\beta, f]
    $$
    for any $[P, \alpha, f]$ and $[P, \beta, f]$ in $\overline{TL_{\vec{F}}}$, and any $a, b \in \mathbb{Z}$. The multiplication on $TL_{\vec{F}}$ is induced by the multiplication in $\overline{TL_{\vec{F}}}$.
\end{definition}

We define the trace function $\mathfrak{tr}: TL_{\vec{F}} \to \mathbb{Q}(A)$, the field of rational functions in $A$, as follows. For a generator $[P, \alpha, f] \in TL_{\vec{F}}$, define
$$
\mathfrak{tr}[P, \alpha, f] = \frac{1}{(-A^2-A^{-2})^{|P|}}tr(\alpha)$$
the normalized trace of $\alpha \in TL_{|P|}$  in the usual sense. Note that $\mathfrak{tr}$ is well defined, because adding $n$ trivial strands means multiplying the trace by $(-A^2-A^{-2})^n$. The trace of a general element in $TL_{\vec{F}}$ can be defined linearly.

For a generator $[P, \alpha, f] \in TL_{\vec{F}}$, define its vertical mirror to be 
$$
[P, \alpha, f]^* = [f^{-1}(P), \alpha^*, f^{-1}]
$$
where $\alpha^*$ is the vertical mirror of $\alpha \in TL_{|P|}$. The vertical mirror operation extends to $TL_{\vec{F}}$ semilinearly by sending coefficient $A$ to $A^{-1}$.

Given two elements $\theta, \gamma \in TL_{\vec{F}}$, define their inner product to be
$$
\langle \theta, \gamma \rangle = \mathfrak{tr}(\gamma^*\theta)
$$

\vspace{5mm}

Note that both the oriented dyadic braid group $B_{\vec{F}}$ and the oriented dyadic Temperley-Lieb monoid $\overline{TL_{\vec{F}}}$ are obtained as $\vec{F}$-equivariant direct limits. We therefore use Proposition \ref{prop:S-functor-induces-SG-functor} to define a map from $B_{\vec{F}}$ to $\overline{TL_{\vec{F}}}$ that generalizes Jones representation $J_n: Br_n \to TL_n$.

Consider the direct category $\mathcal{S}^{\frac{1}{2}}$ with $\vec{F}$-action on it, and two $\vec{F}$-equivariant functors $Br: \mathcal{S}^{\frac{1}{2}} \to \mathcal{M}on$ and $TL: \mathcal{S}^{\frac{1}{2}} \to \mathcal{M}on$ defined in previous sections. We define a natural transformation $\tau$ from $Br$ to $TL$ as follows. For any object $P$ in $\mathcal{S}^{\frac{1}{2}}$, define
$$
\tau_P: Br(P) \to TL(P)
$$
to be $J_{|P|}: Br_{|P|} \to TL_{|P|}$.

\begin{lemma}
    $\tau$ is an $\vec{F}$-equivariant natural transformation.
\end{lemma}

\begin{proof}
    We first check that $\tau$ is indeed a natural transformation, that is, we check the following diagram commutes for any $P \leq Q$.
    \[\begin{tikzcd}
    	{Br_{|Q|}} & {TL_{|Q|}} \\
    	{Br_{|P|}} & {TL_{|P|}}
    	\arrow["{J_{|Q|}}", from=1-1, to=1-2]
    	\arrow["{Br(\iota_P^Q)}", from=2-1, to=1-1]
    	\arrow["{J_{|P|}}"', from=2-1, to=2-2]
    	\arrow["{TL(\iota_P^Q)}"', from=2-2, to=1-2]
    \end{tikzcd}\]
    
    For any $\alpha \in Br_{|P|}$, $(J_{|Q|} \circ Br(\iota_P^Q))(\alpha)$ is obtained by first adding trivial strands to $\alpha$ on $Q \setminus P$, then taking its Kauffman bracket. $(TL(\iota_P^Q) \circ J_{|P|})(\alpha)$ is obtained by first taking the Kauffman bracket of $\alpha$, then adding trivial strands on $Q \setminus P$, then taking Kauffman brackets again. Since taking Kauffman brackets doesn't depend on the order of crossings, the above diagram commutes.

    $\tau$ is $\vec{F}$-equivariant, because $\tau_P = J_{|P|} = J_{|g(P)|} = \tau_{g(P)}$ for any object $P$ and any $g\in \vec{F}$.
\end{proof}

By considering the identity functor $id: \mathcal{S}^{\frac{1}{2}} \to \mathcal{S}^{\frac{1}{2}}$, and the $\vec{F}$-equivariant natural transformation $\tau$ from $Br$ to $TL \circ id$, Corollary \ref{cor:S-functor-induces-equivariant-direct-limit-homomorphism} tells us that $id$ and $\tau$ induce a monoid homomorphism $id_\tau: \displaystyle \lim_{\longrightarrow}{}^{\vec{F}} Br = B_{\vec{F}} \to \displaystyle \lim_{\longrightarrow}{}^{\vec{F}} TL = \overline{TL_{\vec{F}}}$.

Let $i: \overline{TL_{\vec{F}}} \hookrightarrow TL_{\vec{F}}$ be the inclusion.

\begin{definition} \label{def:jones-dyadic-braid-representation}
    Define $J_{\vec{F}}: B_{\vec{F}} \to TL_{\vec{F}}$ to be $i \circ id_\tau$. We say $J_{\vec{F}}$ is the Jones representation of $B_{\vec{F}}$.
\end{definition}

More specifically, $J_{\vec{F}}$ sends $b = [P, \alpha, f] \in B_{\vec{F}}$ to $[P, J_{|P|}(\alpha), f] \in TL_{\vec{F}}$, where $J_{|P|}: Br_{|P|} \to TL_{|P|}$ is the Jones representation of the $|P|^{th}$ braid group. 

Recall that each braid group $Br_n$ can be embedded in $B_{\vec{F}}$ via $\alpha \mapsto [P, \alpha, e]$, where $P$ is any set of $n$ dyadic points. Each Temperley-Lieb algebra $TL_n$ can be embedded in $TL_{\vec{F}}$ similarly. In this sense, we have $J_{\vec{F}}|_{Br_n} = J_n: Br_n \to TL_n$.

Furthermore, we have 
\begin{align}
    \mathfrak{tr}(J_{\vec{F}}(b)) &= \mathfrak{tr}[P, J_{|P|}(\alpha), f] \nonumber \\
    &= \frac{1}{(-A^2-A^{-2})^{|P|}}tr(J_{|P|}(\alpha)) \nonumber \\
    &= \frac{1}{(-A^2-A^{-2})^{|P|}} \langle cl(\alpha) \rangle(A) \label{eq: normalized Kauffman for B_F}
\end{align}

Let $\epsilon$ be the unit element in $TL_{\vec{F}}$, then $\langle J_{\vec{F}}(b)(\epsilon), \epsilon \rangle = \mathfrak{tr}(J_{\vec{F}}(b))$, which gives us the normalized Kauffman bracket of $\mathfrak{cl}(b) $ producing Eq. \ref{eq:dyadic-braid-rep-to-kauffman-bracket}.

\begin{theorem} \label{thm:jones-dyadic-braid-representation-is-unitary}
    $J_{\vec{F}}$ is a unitary representation.
\end{theorem}
\begin{proof}
    Let $\theta, \gamma$ be two elements in $TL_{\vec{F}}$, and let $b = [P, \alpha, f]$ be any element in $B_{\vec{F}}$. Then we have $\langle J_{\vec{F}}(b)\theta,  J_{\vec{F}}(b)\gamma \rangle = \mathfrak{tr}(\gamma^*J_{\vec{F}}(b)^*J_{\vec{F}}(b) \theta)$. Note that
    \begin{align*}
        J_{\vec{F}}(b)^*J_{\vec{F}}(b) = [f^{-1}(P), J_{|P|}(\alpha)^*, f^{-1}][P, J_{|P|}(\alpha), f] = [P, J_{|P|}(\alpha)^*J_{|P|}(\alpha) ,e]
    \end{align*}

    Here $J_{|P|}: Br_{|P|} \to TL_{|P|}$ is the classical Jones representation of the braid group $Br_{|P|}$, so we have $J_{|P|}(\alpha)^*J_{|P|}(\alpha) = J_{|P|}(\alpha^*\alpha) = J_{|P|}(id) = id$. Thus $\langle J_{\vec{F}}(b)\theta,  J_{\vec{F}}(b)\gamma \rangle = \mathfrak{tr}(\gamma^* [P, id, e] \theta) = \mathfrak{tr}(\gamma^* \theta) = \langle \theta,  \gamma \rangle$.
\end{proof}

\subsection{Proof of Theorem \ref{thm:braid-rep-thompson-rep}}

Given $g \in \vec{F}$, let $(\mathcal{I},\mathcal{J})$ be a pair of s.d. partitions representing $g$. Recall that $\mathfrak{b}(\mathbb{H}_{\mathcal{I}}, \mathbb{H}_{\mathcal{J}})$ is the braid obtained by applying the grid-to-braid operation on $(\mathbb{H}_{\mathcal{I}}, \mathbb{H}_{\mathcal{J}})$.

\begin{lemma} \label{lem:braid-writhe-zero}
The braid $\mathfrak{b}(\mathbb{H}_{\mathcal{I}}, \mathbb{H}_{\mathcal{J}})$ has writhe $0$.
\end{lemma}

\begin{proof}
    Recall we can view any grid diagram as a Legendrian link in the standard tight contact $S^3$. In \cite{MR2500576} and \cite[Proposition 4]{MR2602837}  we see that when turning a grid diagram into a braid contact geometrically, we are turning a Legendrian link into its transverse push-off. 
    
    Let $L(\mathbb{H}_{\mathcal{I}}, \mathbb{H}_{\mathcal{J}})$ be the Legendrian link corresponding to the grid $(\mathbb{H}_{\mathcal{I}}, \mathbb{H}_{\mathcal{J}})$, and let $T(\mathbb{H}_{\mathcal{I}}, \mathbb{H}_{\mathcal{J}})$ be the transverse push-off of $L(\mathbb{H}_{\mathcal{I}}, \mathbb{H}_{\mathcal{J}})$. It was shown in \cite[Proposition 5.7]{luo2024halfgriddiagramsthompson} that the Thurston-Bennequin number $tb(L(\mathbb{H}_{\mathcal{I}}, \mathbb{H}_{\mathcal{J}}))=-n$, and the rotation number $rot(L(\mathbb{H}_{\mathcal{I}}, \mathbb{H}_{\mathcal{J}}))=0$, where $2n$ is the size of the grid. Thus, according to \cite{MR2179261} we know the self-linking number $sl(T(\mathbb{H}_{\mathcal{I}}, \mathbb{H}_{\mathcal{J}}))=tb(L(\mathbb{H}_{\mathcal{I}}, \mathbb{H}_{\mathcal{J}}))-rot(L(\mathbb{H}_{\mathcal{I}}, \mathbb{H}_{\mathcal{J}}))=-n$. On the other hand, the braid $\mathfrak{b}(\mathbb{H}_{\mathcal{I}}, \mathbb{H}_{\mathcal{J}})$ represents $T(\mathbb{H}_{\mathcal{I}}, \mathbb{H}_{\mathcal{J}})$, so its writhe minus its number of strands equals the self-linking number, $-n$. Since the braid has $n$ strands, its writhe is zero.
\end{proof}

Then Theorem \ref{thm:braid-rep-thompson-rep} follows immediately.

\begin{proof}[Proof of Theorem \ref{thm:braid-rep-thompson-rep}]
    Given $g \in \vec{F}$, let $(\mathcal{I}, \mathcal{J})$ be an s.d. $n$-partition pair representing $g$.
    
    By Lemma \ref{lem:braid-writhe-zero}, $\mathfrak{b}(\mathbb{H}_{\mathcal{I}}, \mathbb{H}_{\mathcal{J}})$ has writhe $0$. Hence
    \begin{align*}
    \langle J_{\vec{F}}(\iota(g))(\epsilon), \epsilon \rangle &= \frac{1}{(-A^2-A^{-2})^{|\mathcal{E}_+(\mathcal{J})|}} \langle cl(\mathfrak{b}(\mathbb{H}_\mathcal{I}, \mathbb{H}_\mathcal{J})) \rangle(A) \\
    &= \frac{1}{(-A^2-A^{-2})^{|\mathcal{E}_+(\mathcal{J})|-1}} V_{\mathfrak{b}(\mathbb{H}_\mathcal{I}, \mathbb{H}_\mathcal{J})}(A)
    \end{align*}

    On the other hand, we have Eq. \ref{equ: normalized Jones' polynomial for Thompson}:
    $$
    \langle \pi_J(g)(\xi), \xi \rangle = \frac{1}{(-A^2-A^{-2})^{n-1}}V_{L(\mathcal{I}, \mathcal{J})}(A)
    $$

    Recall that $cl(\mathfrak{b}(\mathbb{H}_\mathcal{I}, \mathbb{H}_\mathcal{J}))$ is isotopic to $L(\mathbb{H}_\mathcal{I}, \mathbb{H}_\mathcal{J})$, which is $L(\mathcal{I}, \mathcal{J})$ by Theorem \ref{thm: equivalent oriented links}. Note that $|\mathcal{E}_+(\mathcal{J})| = n$, then we have
    $$
    \langle \pi_J(g)(\xi), \xi \rangle = \langle J_{\vec{F}}(\iota(g))(\epsilon), \epsilon \rangle.
    $$
    

    
\end{proof}

\section{$B_{\vec{F}}$ is finitely generated} 

It is well-known \cite{MR1426438} that s.d. partitions are in one-to-one correspondence with binary trees, where the interior breakpoints of a partition correspond to the vertices (not including leaves) of a binary tree. In this section we will study trees and treat vertices as dyadic numbers.

We continue to use the sign convention on $\mathcal{E}$ defined in Section \ref{sec:half-grid}, so for each binary tree, its vertices are signed dyadic numbers. Let $s(v)$ denote the sign of a vertex $v$.

\begin{lemma} \label{lem:upper-trees}
    For any finite sequence of signs $(\sigma_1, \sigma_2,\cdots, \sigma_n)$, there exists a tree $T$ with $n$ vertices $\{v_1<v_2<\cdots<v_n\}$ such that  $s(v_i)=\sigma_i$ for all $1 \leq i \leq n$ if $\sigma_1=+$, or $s(v_i)=-\sigma_i$ for all $i$ if $\sigma_1=-$. 
\end{lemma}
\begin{proof}
First note that it suffices to prove the case for $\sigma_1=+$, otherwise we consider the sequence $(-\sigma_1, -\sigma_2,\cdots,-\sigma_n )$.

If $\sigma_i=+$ for all $1\leq i \leq n$, then we consider the tree  $T_n^+$ whose  vertex $v_n$ is the root (degree 2 vertex), with all other vertices $v_j$ lying on the left edge of the root (Figure \ref{fig:all-left-tree}).

\begin{figure}
    \centering
    \includegraphics[width=0.5\linewidth]{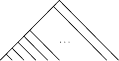}
    \caption{The tree $T^+_n$.}
    \label{fig:all-left-tree}
\end{figure}

Thus we may assume there exists some $i\neq 1$ such that $\sigma_i = -$. Then we prove the statement inductively. The base case is the sequence $(+, -)$, for which we take the tree with a vertex on the right edge of the root. Assume that the statement holds for all sequences of length less than $n$. For any sequence $(\sigma_1, \sigma_2,\cdots, \sigma_n)$, let $j$ be the smallest number such that $\sigma_j = -$. Note that $(-\sigma_j, -\sigma_{j+1}, ..., -\sigma_n)$ is a sequence of length shorter than $n$, with the first sign $+$. We attach its corresponding tree to the right edge of $T_{j-1}^+$, then the resulting tree is the desired one.
\end{proof}

\begin{lemma} \label{lem:lower-tree}
    For any pair of dyadic numbers $a < b$, there exists a tree with vertices $v_1, v_2,..., v_n$ such that for some $1\leq i<j < n$, we have $v_i = a, v_j = b$. Moreover, we can assume that the tree satisfies $s(v_1) = +$, $s(v_2)=-$, $s(v_{i+1})=-s(v_i)$, and  $ s(v_{j+1})=-s(v_j)$.
\end{lemma}

\begin{proof}
    We can easily find a s.d. partition containing breakpoints $a, b$, then consider the corresponding tree with vertices $v_1, v_2, ..., v_n$. Suppose that $v_i = a$ and $v_j = b$. Now we want to fix the signs of $v_2, v_{i+1}, v_{j+1}$.

    To realize that, we just need to add an extra vertex next to each of $v_1$,$v_i$, and $v_j$ lying on the right edge.
\end{proof}

\begin{proposition} \label{prop:two-numbers-to-any-ab}
    For any pair of dyadic numbers $a < b$, there exists $f \in \vec{F}$ such that
    \begin{itemize}
        \item $f(\frac{1}{4}) = a$ and $f(\frac{1}{2})=b$, if $s(a)=s(b)=+$.
        \item $f(\frac{1}{2}) = a$ and $f(\frac{3}{4})=b$, if $s(a)=+, s(b)=-$.
        \item $f(\frac{5}{8}) = a$ and $f(\frac{3}{4})=b$, if $s(a)=s(b)=-$.
        \item $f(\frac{3}{4}) = a$ and $f(\frac{7}{8})=b$, if $s(a)=-, s(b)=+$.
    \end{itemize}
\end{proposition}

\begin{proof}
    We start with the first case that $s(a) = s(b) = +$. It suffices to find two trees $T^{up}$ with vertices $\{v_1^{up}<\cdots<v_n^{up}\}$  and $T^{down}$ with vertices $\{v_1^{down}<\cdots<v_n^{down}\}$ satisfying the following: 
    \begin{enumerate} 
        \item The sign $s(v_k^{up})=s(v_k^{down})$ for each $1\leq k \leq n$.
        \item There exist $1 \leq i,j \leq n$ such that $v_i^{up}=\frac{1}{4}$, $v_j^{up}=\frac{1}{2}$ and $v_i^{down}=a$, $v_j^{down}=b$.
    \end{enumerate}

Then this pair of trees $(T^{up}, T^{down})$ gives us an element $f \in \vec{F}$ sending $\frac{1}{4}$ to $a$ and sending $\frac{1}{2}$ to $b$. We will use the above two lemmas to construct these two trees. 
    
We let $T^{down}$ be the tree given by Lemma \ref{lem:lower-tree}, such that $v_i^{down} = a, v_j^{down} = b$ and $s(v_1^{down})=+, s(v_2^{down})=-, s(v_{i+1}^{down})=-, s(v_{j+1}^{down})=-$. Consider the three sub-sequences separated by $v_i^{down}$ and $v_j^{down}$. They induce three sequences of signs
    \begin{align*}
        & \Sigma_1 = \{s(v_1^{down}) = +, s(v_2^{down}) = -, \cdots, s(v_{i-1}^{down})\}\\
        & \Sigma_2 = \{s(v_{i+1}^{down}) = -, s(v_{i+2}^{down}), \cdots, s(v_{j-1}^{down})\} \\
        & \Sigma_3 = \{s(v_{j+1}^{down}) = -, s(v_{j+2}^{down}), \cdots, s(v_n^{down})\}
    \end{align*}
    
    By Lemma \ref{lem:upper-trees}, we can find a tree $T_1^{up}$ inducing $\Sigma_1$, a tree $T_2^{up}$ inducing $-\Sigma_2$ (because $s(v_{i+1}^{down}) = -$), and a tree $T_3^{up}$ inducing $-\Sigma_3$ (because $s(v_{j+1}^{down}) = -$). Let $T^{up}$ be the tree constructed from $T_1^{up}, T_2^{up}$ and $T_3^{up}$ as Figure \ref{fig:tree-plus-plus} shows. It is not hard to check that the pair $(T^{up}, T^{down})$ satisfies the above two conditions.

    \begin{figure}
        \centering

        \begin{subfigure}{0.45\textwidth}
            \includegraphics[width=\textwidth]{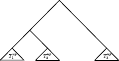}
            \caption{}
            \label{fig:tree-plus-plus}
        \end{subfigure}
        \hfill
        \begin{subfigure}{0.45\textwidth}
        \includegraphics[width=\textwidth]{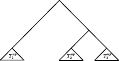}
        \caption{}
        \label{fig:tree-plus-minus}
        \end{subfigure}
        \begin{subfigure}{0.45\textwidth}
        \includegraphics[width=\textwidth]{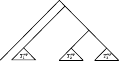}
        \caption{}
        \label{fig:tree-minus-plus}
        \end{subfigure}
        \hfill
        \begin{subfigure}{0.45\textwidth}
        \includegraphics[width=\textwidth]{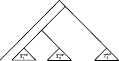}
        \caption{}
        \label{fig:tree-minus-minus}
        \end{subfigure}
        \caption{Trees (A), (B), (C) and (D) correspond to the four possible sign combinations of the points $a$ and $b$.}
    \end{figure}

    The second case can be proved similarly, but with $T_3^{up}$ inducing $\Sigma_3$ instead (because $s(v_{j+1}^{down}) = +$). The tree $T^{up}$ is constructed from $T_1^{up}, T_2^{up}$ and $T_3^{up}$ as Figure \ref{fig:tree-plus-minus} shows.

    The proof for the third and fourth cases is slightly different from the first two. After we determine the tree $T^{down}$, we let the first sequence of signs instead be
    $$
    \Sigma_1' = \{s(v_2^{down}) = -, \cdots, s(v_{i-1}^{down})\}
    $$
    without $s(v_1^{down})$. Then consider the three sequences $\Sigma_1', \Sigma_2$ and $\Sigma_3$.

    In the third case, we find a tree $T_1^{up}$ inducing $-\Sigma_1'$, a tree $T_2^{up}$ inducing $\Sigma_2$ and a tree $T_3^{up}$ inducing $\Sigma_3$. The tree $T^{up}$ is given by Figure \ref{fig:tree-minus-plus}.

    In the fourth case, the tree $T_3^{up}$ induces $-\Sigma_3$ instead, and the tree $T^{up}$ is given by Figure \ref{fig:tree-minus-minus}. This completes the proof.
\end{proof}

Now we are ready to prove that the oriented dyadic braid group $B_{\vec{F}}$ is finitely generated.
\begin{proof}[Proof of Theorem \ref{thm: B_F is finitely generated}]
    Recall that $B_{\vec{F}}$ has braid generators $\sigma_{ab}$ for any two dyadic rational numbers $a,b$, and $\vec{F}$ generators $g_i$ for any nonnegative integer $i.$ We claim that $B_{\vec{F}}$ can be finitely generated by
    $$
    \sigma_{\frac{1}{4} \frac{1}{2}}, \sigma_{\frac{1}{2} \frac{3}{4}}, \sigma_{\frac{5}{8} \frac{3}{4}}, \sigma_{\frac{3}{4} \frac{7}{8}}, g_0, g_1, g_2
    $$
    
    We just need to show that any $\sigma_{ab}$ and $g_i$ can be generated by the above $7$ generators.

    By \cite{MR3565428}, $\vec{F}$ is isomorphic to the $3$rd Brown-Thompson group $F_3$, which has presentation
    $$
    \langle y_0, y_1, y_2, \cdots \mid y_n y_l = y_l y_{n+2}, \textup{for $l < n$} \rangle
    $$

    Note that these generators and relations are exactly the $\vec{F}$ generators $g_i$ and the $\vec{F}$ relations in Definition \ref{def:oriented-dyadic-braid-group}. Since $F_3$ can be generated by the first three generators $y_0, y_1, y_2$, any $\vec{F}$ generator $g_i$ can be generated by $g_0, g_1, g_2$.

    For any braid generator $\sigma_{ab}$ with $s(a)=s(b)=+$, there exists $f \in \vec{F}$ such that $f(\frac{1}{4})=a$ and $f(\frac{1}{2})=b$ by Proposition \ref{prop:two-numbers-to-any-ab}. From the previous argument, $f$ can be viewed as an element in $B_{\vec{F}}$ which can be generated by $g_0, g_1, g_2$. Hence, by the braid $\vec{F}$ relations, we have $\sigma_{ab} = \sigma_{f(\frac{1}{4})f(\frac{1}{2})} = f \sigma_{\frac{1}{4}\frac{1}{2}} f^{-1}$, so $\sigma_{ab}$ is generated by $g_0,g_1,g_2,\sigma_{\frac{1}{4}\frac{1}{2}}$.

    The proof is similar for $a, b$ with other sign combinations. 
\end{proof}

\bibliographystyle{alpha}
\bibliography{bib}

\end{document}